\documentclass[12pt,reqno]{amsart}
\usepackage{fullpage}

\usepackage{graphicx}
\usepackage{color}
\usepackage{subfigure}
\usepackage{amssymb}
\usepackage{amsmath}
\usepackage{colonequals}
\usepackage{hyperref}

\newtheorem{theorem}{Theorem}[section]
\newtheorem{lemma}[theorem]{Lemma}

\newtheorem{definition}[theorem]{Definition}

\newtheorem{proposition}[theorem]{Proposition}  %added
\newtheorem{corollary}[theorem]{Corollary}  %added
\newtheorem{conj}[theorem]{Conjecture}  %added
\newtheorem{question}[theorem]{Question}  %added

\newtheorem{remark}[theorem]{Remark}

\renewcommand{\subset}{\subseteq}

\renewcommand{\epsilon}{\varepsilon}
\renewcommand{\nu}{v}
\def\eps{{\varepsilon}}

\newcommand{\ignore}[1]{}
\newcommand{\abs}[1]{\left|#1\right|}                   % Absolute value notation
\newcommand{\absf}[1]{|#1|}                             % small absolute value signs
\newcommand{\vnorm}[1]{\left\|#1\right\|}    % norm notation
\newcommand{\vnormf}[1]{\|#1\|}                         % norm notation, forced to be small
\newcommand{\sign}[1]{\mbox{sign}#1}
\newcommand{\Z}{\mathbb{Z}}                             % Blackboard notation

\newcommand{\N}{\mathbb{N}}
\newcommand{\E}{\mathbb{E}}
\renewcommand{\P}{\mathbb{P}}

\newcommand{\R}{\mathbb{R}}

\newcommand{\Q}{\mathbb{Q}}
\newcommand{\cH}{\mathcal{H}}
\newcommand{\cC}{\mathcal{C}}
\newcommand{\vp}{\varphi}
\newcommand{\tOmega}{{\tilde{\Omega}}}
\newcommand{\tmu}{{\tilde{\mu}}}
\newcommand{\lj}{{L^{1}(\Omega,\mu)}}
\newcommand{\ld}{{L^{2}(\Omega,\mu)}}
\newcommand{\lp}{{L^{p}(\Omega,\mu)}}
\newcommand{\li}{{L^{\infty}(\Omega,\mu)}}
\newcommand{\ltj}{{L^{1}(\tOmega,\tmu)}}
\newcommand{\ltd}{{L^{2}(\tOmega,\tmu)}}
\newcommand{\ltp}{{L^{p}(\tOmega,\tmu)}}
\newcommand{\lti}{{L^{\infty}(\tOmega,\tmu)}}
\newcommand{\1}{{\bf{1}}}
\newcommand{\embolden}[1]{{#1}}

\renewcommand{\sign}[1]{\mathrm{sign}#1}

\numberwithin{equation}{section}

\begin{document}

\title{Strong Contraction and Influences in Tail Spaces}


%    author one information
\author[S. Heilman]{Steven Heilman}
\address{UCLA Department of Mathematics, Los Angeles, CA 90095-1555}
\curraddr{}
\email{heilman@cims.nyu.edu}
\thanks{S.\,H. was supported by NSF Graduate Research Fellowship DGE-0813964 and a Simons-Berkeley Research Fellowship.  Part of this work was completed while S.\,H. was visiting the Network Science and Graph Algorithms program at ICERM}

%    author two information
\author{Elchanan Mossel}
\address{Department of Statistics, University of Pennsylvania, Philadelphia, PA 19104 and
Departments of Statistics and Computer Science, U.C. Berkeley, Berkeley CA 94720}
\email{mossel@wharton.upenn.edu}
\curraddr{}
\thanks{E.\,M. was supported by NSF grant DMS-1106999, NSF Grant CCF 1320105 and DOD ONR grant N000141110140 and grant 328025 from the Simons foundation}

\author{Krzysztof Oleszkiewicz}
\address{Institute of Mathematics, University of Warsaw, Banacha 2, 02-097 Warszawa, Poland}
\email{koles@mimuw.edu.pl}
\thanks{K.\,O. was supported by NCN grant DEC-2012/05/B/ST1/00412.
Part of this work was carried out while the authors were visiting the Real Analysis in Computer Science program
at the Simons Institute for the Theory of Computing.}

%    \subjclass is required.


%\keywords{Strong contraction, influences, tail spaces, Kahn-Kalai-Linial, semi-group, Poincare inequality, Talagrand inequality, Harper inequality}

\date{}

\dedicatory{}

%    Abstract is required.
\begin{abstract}
We study contraction under a Markov semi-group and influence bounds for functions in $L^2$ tail spaces, i.e. functions all of whose low level Fourier coefficients vanish. It is natural to expect that certain analytic inequalities are stronger for such functions than for general functions in $L^2$. In the positive direction we prove an $L^{p}$ Poincar\'{e} inequality and moment decay estimates for mean $0$ functions and for all $1<p<\infty$, proving the degree one case of a conjecture of Mendel and Naor as well as the general degree case of the conjecture when restricted to Boolean functions.
%The original motivation for these Poincar\'{e} inequalities was a simplified construction of expander graphs which have a spectral gap with respect to any uniformly convex space.
In the negative direction, we answer negatively two questions of Hatami and Kalai concerning extensions of the Kahn-Kalai-Linial and Harper Theorems to tail spaces. That is, we construct a function $f\colon\{-1,1\}^{n}\to\{-1,1\}$ whose Fourier coefficients vanish up to level $c \log n$, with all influences bounded by $C \log n/n$ for some constants $0<c,C< \infty$.  We also construct a function $f\colon\{-1,1\}^{n}\to\{0,1\}$ with nonzero mean whose remaining Fourier coefficients vanish up to level $c' \log n$, with the sum of the influences bounded by $C'(\mathbb{E}f)\log(1/\mathbb{E}f)$ for some constants $0<c',C'<\infty$.
\end{abstract}

\maketitle

\section{Introduction}

%*** changed X to R
Consider the uniform measure on $\{-1,1\}^n$. Any $f\colon\{-1,1\}^{n}\to \R$ can be written as $f=\sum_{S\subset\{1,\ldots,n\}}\widehat{f}(S)W_{S}$, where for all $x=(x_{1},\ldots,x_{n})\in\{-1,1\}^{n}$, $W_{S}(x)\colonequals\prod_{i\in S}x_{i}$ and $\widehat{f}(S)\colonequals 2^{-n}\sum_{x\in\{-1,1\}^{n}}f(x)W_{S}(x)$.  For any $t\geq0$, define $P_{t}f\colonequals\sum_{S\subset\{1,\ldots,n\}}e^{-t\abs{S}}\widehat{f}(S)W_{S}$, and define $Lf\colonequals\sum_{S\subset\{1,\ldots,n\}}\abs{S}\widehat{f}(S)W_{S}$.

Our interest in this paper is in tail spaces. For the case of the uniform measure on $\{-1,1\}^n$, we are interested in the linear subspace of all functions satisfying $\widehat{f}(S) = 0$ for all $S\subset\{1,\ldots,n\}$ with $|S| \leq k$. Our interest in understanding such functions follows recent conjectures by Mendel and Naor and by Hatami and Kalai.

\subsection{Heat Smoothing}

%*** Changed "general" to  "a general"
In their study of a general notion of expander (with respect to all uniformly convex spaces), Mendel and Naor made the following conjecture:

\begin{conj}[\embolden{Heat Smoothing}]\label{conj1}\cite[Remark 5.5]{mendel12}
Let $1<p<\infty$.  Let\\ $f\colon\{-1,1\}^{n}\to\R$ with $\E f W_{S}=0$ for all $S\subset\{1,\ldots,n\}$ with $\abs{S}<k$.  Then
\begin{equation}\label{zero0}
\forall t>0,\quad\vnormf{P_{t}f}_{p}\leq e^{-tkc(p)}\vnorm{f}_{p}.
\end{equation}
%\begin{equation}\label{zero0.1}
%\vnorm{Lf}_{p}\geq c(p)k\vnorm{f}_{p}.
%\end{equation}
\end{conj}

In our main result we prove a special case of their conjecture for $k=1$:
\begin{theorem}[\embolden{Heat Smoothing}]\label{main_boolean}
For every $p \in (1,\infty)$ and every $f: \{-1,1\}^n \rightarrow \R$ with $\E f=0$, for every $t>0$,
$$
\vnormf{P_{t}f}_{p} \leq \exp\left(-\frac{(2p-2)t}{(p^{2}-2p+2)}\right)\cdot \|f\|_{p}.
$$
\end{theorem}

%*** changed 'this proof of the' to 'the proof of this'
%*** added $L_{2}$ to the poincare inequality, for clarity
The proof of this theorem covers all Markov operators satisfying an $L^{2}$ Poincar\'e inequality. We also show that if we restrict to $\{-1,0,1\}$-valued functions, then \eqref{zero0} always holds.

\begin{theorem}[\embolden{Conjecture \ref{conj1} for $\{-1,0,1\}$-valued functions}]\label{thm13}
Let $1<p<\infty$ and let $f\colon\{-1,1\}^{n}\to\{-1,0,1\}$ with $\E f W_{S}=0$ for all $S\subset\{1,\ldots,n\}$ with $\abs{S}<k$.  Then for all $t>0$,
\begin{equation}\label{one3}
%***STRENGTHENED "\vnormf{P_{t}f}_{p}\leq e^{-tk\min\left(\frac{2(p-1)}{p},\frac{1}{p-1}\right)}\vnorm{f}_{p}." to
\vnormf{P_{t}f}_{p}\leq e^{-2tk\min\left(\frac{p-1}{p},\frac{1}{p}\right)}\vnorm{f}_{p}.
\end{equation}
\end{theorem}

%*** removed 'C=1', since C is not yet defined
The constant in Theorem \ref{thm13} for $k=1$, which comes from an application of H\"{o}lder's inequality, is strictly worse than that of Theorem \ref{main_boolean}.
Again the proof of Theorem \ref{thm13} extends to cover $P_{t}$ being any symmetric Markov semigroup as long as  $f\colon\Omega\to\{-1,0,1\}$ satisfies $\vnorm{P_{t}f}_{2}\leq e^{-tk}\vnorm{f}_{2}$.

Our results in Theorem~\ref{main_boolean} and Theorem~\ref{thm13} should be compared
to the following result of Mendel-Naor below, which they attributed to P. A. Meyer \cite{meyer84}.
%For $p\geq2$ and $\Omega=\{-1,1\}^{n}$, Theorem \ref{main_boolean}  to the much stronger estimate of Mendel-Naor below, which they attributed to P. A. Meyer \cite{meyer84}.  For $S\subset\{1,\ldots,n\}$ and $x=(x_{1},\ldots,x_{n})\in\{-1,1\}^{n}$, define $W_{S}(x)\colonequals\prod_{i\in S}x_{i}$.
%However, the proof of Mendel-Naor, which uses H\"{o}lder's inequality and hypercontractivity,

\begin{theorem}\cite[Lemma 5.4]{mendel12}\label{thm20}
Let $2\leq p<\infty$.  Then there exists $c(p)>0$ such that the following holds.  Let $f\colon\{-1,1\}^{n}\to\R$ with $\E f W_{S}=0$ for all $\abs{S}<k$.  Then
$$
\forall\,t>0,\quad\vnormf{P_{t}f}_{p}\leq e^{-k\min(t,t^{2})c(p)}\vnorm{f}_{p}.
$$
$$
\vnormf{Lf}_{p}\geq c(p)\sqrt{k}\vnorm{f}_{p}.$$
\end{theorem}
The second inequality can be considered a ``higher-order'' Poincar\'{e} inequality, and it follows from the first by writing $f=\int_{0}^{\infty}e^{-tL}Lfdt$ and then applying the $L^{p}(\{-1,1\}^{n})$ triangle inequality.

%*** changed setup to setting
One should also compare our results to the following result of Hino (in a much more general setting) that is also briefly mentioned at the end of the proof of Theorem 1 in \cite{meyer84}.

\begin{theorem}\cite[Theorem 3.6(ii)b]{hino00}\label{thm9}
Let $1<p<\infty$.  Then there exists $\infty > M(n),\delta(n)>0$ such that, for any $f\colon \{-1,1\}^n \to\R$ with $\E f=0$, for any $t>0$,
$$\vnorm{P_{t}f}_{p}\leq M(n)e^{-\delta(n) t}\vnorm{f}_{p}$$
\end{theorem}
The dependence of the constants $M(n)$ and $\delta(n)$ on the dimension makes this inequality weaker than the previous two in settings where dimension independent inequalities are desired.

%, then \eqref{one3} holds for $f$ by H\"{o}lder's inequality, as shown in the proof of Theorem \ref{thm13}.
%\end{remark}

\subsection{Poincar\'{e} Inequalities}
This heat smoothing estimate in Theorem~\ref{main_boolean} is equivalent to the following Poincar\'{e} inequality.
\begin{theorem}[\embolden{Poincar\'{e} Inequality}]\label{main2_boolean}
Under the above assumptions for every $p \in (1,\infty)$ and every $f: \{-1,1\}^n \rightarrow \R$ with $\E f=0$ there is
$$
\E \abs{f}^{p-1}\sign(f)Lf\geq\frac{2p-2}{(p^{2}-2p+2)}\cdot\E\abs{f}^{p}.
$$
\end{theorem}

The usual Poincar\'{e} inequality corresponds to the case $p=2$ of Theorem \ref{main2_boolean}.  Theorem \ref{main2_boolean}
should be contrasted with Beckner's Poincar\'{e} inequality.
\begin{theorem}\cite{beckner89}\label{thm15}
Let $f\colon\{-1,1\}^{n}\to\R$.  For all $1\leq p\leq2$,
$$
(2-p)\E fLf\geq\E\abs{f}^{2}-(\E\abs{f}^{p})^{2/p}.
$$
\end{theorem}
Specifically, Beckner notes that, for $t>0$ with $e^{-2t}=p-1$, $(2-p)\E fLf\geq\E\abs{f}^{2}-\E\abs{P_{t}f}^{2}$ by Fourier analysis.  He then adds the hypercontractive inequality \cite{bonami70,nelson73,gross75} to this inequality to prove Theorem \ref{thm15}.  However, Theorem \ref{main2_boolean} does not seem to follow from hypercontractivity
%even when $\{-1,1\}^{n}$,
so we need to apply different methods.

%A family of ``classical'' $L_{p}$ Poincar\'{e} inequalities are proved on the hypercube $\{-1,1\}^{n}$ in \cite{efraim08}, where the connection to Riesz transforms \cite{pisier88} is emphasized.

%*** Changed subsection title to incorporate Harper's theorem
\subsection{The KKL, Talagrand and Harper theorems in Tail Spaces}

%*** Changed this paragraph to incorporate Harper's theorem
The KKL Theorem and its strengthening by Talagrand are two of the most fundamental theorems in the theory of Boolean functions.  Harper's theorem is an edge-isoperimetric inequality on the hypercube.  Recent questions by Hatami and Kalai asked if the KKL and Harper theorems could be improved for functions in tail spaces. It is natural to ask the same question for Talagrand's theorem.

We recall some standard definitions.

\begin{definition}[\embolden{Influences}]\label{infdef}
Let $f\colon\{-1,1\}^{n}\to\{-1,1\}$ and let $i\in\{1,\ldots,n\}$.  Define the $i$'th
 {\em influence} $I_i(f)\in\R$ of a function $f : \{-1,1\}^n \to \{-1,1\}$ by
 \[
 I_i(f) \colonequals
 P[f(x_1,\ldots,x_{i-1},x_i,x_{i+1},\ldots,x_n) \neq f(x_1,\ldots,x_{i-1},y,x_{i+1},\ldots,x_n)],
 \]
 where $x_i,y$ are i.i.d. uniform random variables on $\{-1,1\}$ for all $i=1,\ldots,n$.
\end{definition}

%\begin{definition}[\embolden{Walsh functions}]\label{walshdef}
%Let $x=(x_{1},\ldots,x_{n})\in\{-1,1\}^{n}$ and let $S\subset\{1,\ldots,n\}$.  Define $W_{S}\colon\{-1,1\}^{n}\to\{-1,1\}$ by $W_{S}(x)\colonequals\prod_{i\in S}x_{i}$.
%\end{definition}

%Theorem \ref{thm4} gives a slightly stronger form of the Kahn-Kalai-Linial theorem for Boolean functions $f\colon\{-1,1\}^{n}\to\{-1,1\}$.  However, this strengthened inequality has some shortcomings, which we now discuss.

Since the range of a Boolean function is restricted to $\{-1,1\}$, its Fourier coefficients should satisfy some constraints that general real-valued functions with $\vnorm{f}_{2}=1$ do not satisfy.  For instance, the influences of a Boolean function could be slightly larger than expected.  For example, the non-Boolean function $f=(n(n-1)/2)^{-1/2}\sum_{S\subset\{1,\ldots,n\}\colon\abs{S}=2}W_{S}$ satisfies $\vnorm{f}_{2}=1$, where $I_{i}f=2/n$ for all $i=1,\ldots,n$.  At the opposite extreme, the Boolean function $f=W_{\{1,\ldots,n\}}$ satisfies $\vnorm{f}_{2}=1$, where $I_{i}f=1$ for all $i=1,\ldots,n$.  With these examples in mind, we may be led to believe that Boolean functions have larger influences than arbitrary functions with $\vnorm{f}_{2}=1$.  Indeed, Ben-Or and Linial proved the following Proposition, and they conjectured that their bound on influences was the best possible.
%maybe some Fourier coefficients of a Boolean function lack some symmetry, whereas a real-valued function with $\vnorm{f}_{2}=1$ can have all Fourier coefficients equal to each other. %Indeed, by selecting a few real numbers to be the Fourier coefficients of some function, it seems difficult in practice to get the outputs of $\{-1,1\}$ for any input from $\{-1,1\}^{n}$.

\begin{proposition}\label{prop1}
\cite[Theorem 3]{benor89}
There exists a universal constant $c'>0$ and there exists a Boolean function $f \colon\{-1,1\}^n \to \{-1,1\}$ with $\E f=0$ such that $\max_{i=1,\ldots,n} I_i(f) \leq c'(\log n)/n$.
\end{proposition}

Kahn, Kalai and Linial then showed that the influence bound in Proposition \ref{prop1} is in fact the best possible, thereby proving the conjecture of Ben-Or and Linial.

%*** Kahn-Kalai-Linial changed to KKL
\begin{theorem}[\embolden{KKL}]\label{kkl}\cite[Theorem 3.1]{kahn88}
There exists a universal constant $c>0$ such that, for any $f\colon\{-1,1\}^{n}\to\{-1,1\}$, $\max_{i=1,\ldots,n}I_{i}(f)\geq c(\E(f-\E f)^{2})(\log n)/n$.
\end{theorem}

%*** CHANGED "one larger influence" to "a larger influence"
If a function $f\colon\{-1,1\}^{n}\to\{-1,1\}$ not only has mean zero, but it also has many Fourier coefficients which are zero, it similarly seems that even more special structure should exist within the Fourier coefficients of $f$.  That is, perhaps this function should have a larger influence than a mean zero function.  Hatami and Kalai therefore asked the following question, which would improve upon Theorem \ref{kkl}.
\begin{question}\label{q1}
Suppose $k=k(n)\to\infty$ as $n\to\infty$.  Does there exist $\omega(k)>0$ such that $\omega(k)\to\infty$ as $k\to\infty$, such that the following statement holds?
Let $f\colon\{-1,1\}^{n}\to\{-1,1\}$ with $\E fW_{S}=0$ for all $S\subset\{1,\ldots,n\}$ with $\abs{S}\leq k$.  Then $\max_{i=1,\ldots,n}I_{i}f\geq((\log n)/n)\cdot\omega(k)$.
\end{question}

Hatami speculated that a positive answer to the question above may help in proving the Entropy Influence Conjecture.
%***ADDED "we" in "Here prove"
Here we prove
that the answer to the question is negative by showing that
 \begin{theorem}[Question \ref{q1} for $k=\log n$]\label{thm30}
There exists $0<C,c<\infty$ such that, for infinitely many $n\in\N$, there exists $f\colon\{-1,1\}^{n}\to\{-1,1\}$ with $\E fW_{S}=0$ for all $S\subset\{1,\ldots,n\}$ with $\abs{S}\leq c \log n$ such that $\max_{i=1,\ldots,n}I_{i}f\leq C(\log n)/n$.
\end{theorem}

In other words, there is a phase transition for the maximum influence of Boolean functions with vanishing Fourier coefficients.
This phase transition occurs when we require the first $k(n)$ Fourier coefficients to vanish where $k(n)/\log n$ is either bounded or unbounded, as $n\to\infty$.
We note that the functions constructed in Theorem~\ref{thm30}
%***REPLACED "does not" by
do not
provide a counter example to the Entropy Influence conjecture as their
%***REMOVED capital letter in "Entropy"
entropy
is of the same order as
%***ADDED "for"
for
the standard Tribes function.

We also note that  if $k = g(n) \log n$, where $g(n) \to \infty$ then it is trivial to improve the KKL estimate since
\[
\sum_{i=1}^{n}I_{i}f=\sum_{S\subset\{1,\ldots,n\}}\abs{S}\absf{\widehat{f}(S)}^{2}
\]
which implies
$\max_{i=1,\ldots,n}I_{i}f\geq g(n)(\log n)/n$.

%*** ADDED the following few paragraphs about Harper's inequality.
With a similar motivation to Question \ref{q1}, Kalai also asked whether or not the following isoperimetric inequality could be improved.
\begin{theorem}[\embolden{Harper's Inequality}]\label{harper}\cite[Theorem 2.39]{odonnell14},\cite{harper64}
For any $f\colon\{-1,1\}^{n}\to\{0,1\}$, $\sum_{i=1}^{n}I_{i}f\geq(2/\log 2)(\E f)\log(1/\E f)$.
\end{theorem}
To see the isoperimetric content of Theorem \ref{harper}, we consider the hypercube $\{-1,1\}^{n}$ as the vertices of a graph, where an edge connects $(x_{1},\ldots,x_{n}),(y_{1},\ldots,y_{n})\in\{-1,1\}^{n}$ if and only if $\abs{\{i\in\{1,\ldots,n\}\colon x_{i}\neq y_{i}\}}=1$.  Then $(1/n)\sum_{i=1}^{n}I_{i}f$ is equal to the fraction of edges of $\{-1,1\}^{n}$ between the sets $\{x\in\{-1,1\}^{n}\colon f(x)=0\}$ and $\{x\in\{-1,1\}^{n}\colon f(x)=1\}$.  And the quantity $(\E f)\log(1/\E f)$ measures the volume of the set $\{x\in\{-1,1\}^{n}\colon f(x)=1\}$.
%for this function f, the left side is 1, and the mean is 1/2, so the right side is 1.

If $f(x_{1},\ldots,x_{n})\colonequals (x_{1}+1)/2$, then equality nearly holds in Theorem \ref{harper}.  Note that, in this case, $f$ has Fourier coefficients only of degrees zero and one.  It therefore seems sensible that, if $f$ has only Fourier coefficients of higher order, then $f$ will oscillate, so the perimeter of its level sets should be much larger than the volume of its level sets.  Kalai therefore asked if the constant $2$ in Theorem \ref{harper} would become large when a large number of Fourier coefficients of the function are zero.

\begin{question}\label{q2}
Suppose $k=k(n)\to\infty$ as $n\to\infty$.  Does there exist $\omega(k)>0$ such that $\omega(k)\to\infty$ as $k\to\infty$, such that the following statement holds?
Let $f\colon\{-1,1\}^{n}\to\{0,1\}$ with $\E fW_{S}=0$ for all $S\subset\{1,\ldots,n\}$ with $1\leq\abs{S}\leq k$.  Then $\sum_{i=1}^{n}I_{i}f\geq\omega(k)(\E f)\log(1/\E f)$.
\end{question}

A simplification of the function from Theorem \ref{thm30} shows that Question \ref{q2} has a negative answer.

 \begin{theorem}[Negative Answer to Question \ref{q2}]\label{thm31}
There exists $0<C,c<\infty$ such that, for infinitely many $n\in\N$, there exists $f\colon\{-1,1\}^{n}\to\{0,1\}$ with $\E fW_{S}=0$ for all $S\subset\{1,\ldots,n\}$ with $1\leq\abs{S}\leq c \log n$ such that $\sum_{i=1}^{n}I_{i}f\leq C(\E f)\log(1/\E f)$.
\end{theorem}

A less trivial argument
%***REMOVED using Theorem~\ref{thm20}
allows to extend Talagrand's theorem to tail spaces.
\begin{theorem}[\embolden{Talagrand Inequality for Tail Space}]\label{thm4}  Let $k\geq1$.  Let $f\colon\{-1,1\}^{n}\to\R$ with $\E f W_{S}=0$ for all $S\subset\{1,\ldots,n\}$ with $\abs{S}<k$.  For $i=1,\ldots,n$, define $D_{i}f(x)\colonequals [f(x_{1},\ldots,x_{n})-f(x_{1},\ldots,x_{i-1},-x_{i},x_{i+1},\ldots,x_{n})]/2$.  Then
\begin{equation}\label{two44}
%\begin{aligned}
%&f\in L_{p}^{\geq k}(\gamma_{n})\quad\Longrightarrow\\
%&\quad\int\abs{f}^{2}d\gamma_{n}
%\leq\sum_{i=1}^{n}\vnorm{\partial f/\partial x_{i}}_{L_{2}(\gamma_{n})}^{2}\int_{0}^{1}2u^{1+k}
%\left(\frac{\vnorm{\partial f/\partial x_{i}}_{L_{1}(\gamma_{n})}}{\vnorm{\partial f/\partial x_{i}}_{L_{2}(\gamma_{n})}}\right)^{1-u}du.
%\end{aligned}
\begin{aligned}
\E f^{2}
\leq \frac{3}{k}\sum_{i=1}^{n} \E(D_{i}f)^{2}/\max(1, \log(\| D_{i}f\|_{2}/\|D_{i}f\|_{1+e^{-2/k}})) \leq8\sum_{i=1}^{n} \frac{\E (D_{i}f)^{2}}{k+\log \left( \| D_{i}f\|_{2}/\| D_{i}f\|_{1}\right)}.
\end{aligned}
\end{equation}
\end{theorem}
Note that the usual form of Talagrand's inequality is obtained by setting $k=1$ and substituting $f-\E f$ in place of $f$ on the left side of \eqref{two44} (which is redundant for $k\geq1$).

We note that while in principle, Theorem \ref{thm4} may indicate that the answer to Question \ref{q1} is positive, since the usual Talagrand Inequality implies the Kahn-Kalai-Linial Theorem \ref{kkl}.  However, the improvement of Theorem \ref{thm4} over Theorem \ref{kkl} only occurs for $k$ of the form $k=g(n)\log n$ where $g(n)\to\infty$ as $n\to\infty$.
%Yet, for $k=g(n)\log n$, Theorem \ref{kkl} can be improved simply from the formula $\sum_{i=1}^{n}I_{i}f=\sum_{S\subset\{1,\ldots,n\}}\abs{S}\absf{\widehat{f}(S)}^{2}$.  Then we immediately get $\max_{i=1,\ldots,n}I_{i}f\geq g(n)(\log n)/n$, so that Question \ref{q1} is true.  It is therefore surprising that, for $k=\log n$, Question \ref{q1} does not hold, as we now show.

\subsection{General Setting}
Theorem~\ref{main_boolean} and Theorem~\ref{main2_boolean} are proven in the following general setting:
\begin{equation}\label{ten0}
\begin{cases}
&\bullet \,\,(\Omega, 2^{\Omega}, \mu)\mbox{ is a finite probability space.}\\
&\bullet \,\,(P_{t})_{t \geq 0}\mbox{ is a symmetric Markov semigroup on }L^{2}(\Omega, \mu)\\
&\bullet \,\,(P_{t})_{t \geq 0}\mbox{ has generator }L=-\frac{\mathrm{d}}{\mathrm{d} t}P_{t}\Big|_{t=0^{+}}.\\
&\bullet \,\,\E\mbox{ denotes the expectation with respect to the invariant measure }\mu.\\
&\bullet \,\,\| f\|_{p}\mbox{ will stand for }\left(\E |f|^{p}\right)^{1/p}.
\end{cases}
\end{equation}
We assume additionally that $L$ satisfies the Poincar\'e inequality with a positive constant $C$, i.e.
\begin{equation}\label{ten1}
\E f^2- (\E f)^{2} \leq C \cdot\E fLf,
\end{equation}
for every $f:\Omega \rightarrow \R$, or equivalently,
$\E (P_{t}f)^{2} \leq e^{-2t/C}\cdot \E f^{2}$ for every $t \geq 0$ and every mean-zero $f$.
Theorem \ref{main_boolean} is a special case of the following theorem:
\begin{theorem}[\embolden{Heat Smoothing}]\label{main}
Assume that \eqref{ten0} and \eqref{ten1} hold.  Then for every $p \in (1,\infty)$ and every $f: \Omega \rightarrow \R$ with $\E f=0$, for every $t>0$,
$$
\vnormf{P_{t}f}_{p} \leq \exp\left(-\frac{(2p-2)t}{(p^{2}-2p+2)C}\right)\cdot \|f\|_{p}.
$$
\end{theorem}
%*** Added a comment about Cattiaux's proof

Theorem~\ref{main2_boolean} is a special case of the following result.
\begin{theorem}[\embolden{Poincar\'{e} Inequality}]\label{main2}
Assume that \eqref{ten0} and \eqref{ten1} hold.  Then for every $p \in (1,\infty)$ and every $f\colon \Omega \rightarrow \R$ with $\E f=0$ there is
$$
\E \abs{f}^{p-1}\sign(f)Lf\geq\frac{2p-2}{(p^{2}-2p+2)C}\cdot\E\abs{f}^{p}.
$$
\end{theorem}

Theorem \ref{main2} is equivalent to Theorem \ref{main}.  Below, we will first prove Theorem \ref{main2} and then deduce Theorem \ref{main} as a consequence.

For $\Omega=\R$, $p=4$ and $d\mu=e^{-x^{2}/2}dx/\sqrt{2\pi}$, with $L$ being the generator of the Ornstein-Uhlenbeck semigroup, Theorem \ref{main2} was proven by P. Cattiaux, as noted in \cite{mendel12}.

After the proof of Theorem~\ref{main} we briefly discuss how Theorem~\ref{main} and Theorem~\ref{main2} can be extended to infinite spaces.

\subsection{Organization}

We prove Theorem \ref{main} in Section \ref{secpoincare}.  Theorem \ref{main2} is then derived as a Corollary in Section \ref{secheat}, where Theorem \ref{thm13} is also shown.  A complex interpolation proof of Theorem \ref{main} is given in Section \ref{secnaz}, albeit with worse constants.  A semigroup proof of Theorem \ref{main} is given in Section \ref{secsemi}. Theorem \ref{thm4} is proven in Section \ref{sectal}, and Theorem \ref{thm30} is proven in Section \ref{sechatami}.
%Finally, we make a brief comment about Theorems \ref{main} and \ref{main2} for infinite spaces $\Omega$ in Section \ref{secinf}.

\section{Poincar\'{e} Inequalities}\label{secpoincare}
In this section we prove Theorem~\ref{main} and Theorem~\ref{main2}.

%[may only need to use \eqref{three5} in statement of the Lemma; make others claims along the way?]

Let $x\in\R$.  In what follows, we use the standard notation $x_{+}\colonequals\max(x,0)$ and $x_{-}\colonequals\max(-x,0)$, so that $x=x_{+}-x_{-}$ and $\abs{x}^{p}=x_{+}^{p}+x_{-}^{p}$ for any $x\in\R$ and $p>0$.  Also, for $s>0$ we will denote by $\phi_{s}$ the function $\phi_{s}(x)\colonequals\sign(x)\cdot\abs{x}^{s}$, so that $\phi_{s}(x)=x_{+}^{s}-x_{-}^{s}$ for every $x\in\R$.

\begin{lemma}\label{lemma10}
Let $p>1$, and let $X$ be a real random variable with $\E\abs{X}^{p}<\infty$ and $\E X =0$.  For every $p\in(1,\infty)\setminus\{2\}$,

\begin{equation}\label{three5}
(\E X_{+}^{p/2})^{2}
+(\E X_{-}^{p/2})^{2}
+(\E X_{+}^{p/2})^{-\frac{2}{p-2}}
(\E X_{-}^{p/2})^{\frac{2p-2}{p-2}}
+(\E X_{-}^{p/2})^{-\frac{2}{p-2}}
(\E X_{+}^{p/2})^{\frac{2p-2}{p-2}}
\leq\E \abs{X}^{p}.
\end{equation}
\end{lemma}

\begin{proof}
Since $\E X=0$,
\begin{equation}\label{three6}
\E X_{+}
=\E X_{-}=(1/2)\E\abs{X}.
\end{equation}

%use q=(2p-2)/p, then 1/q'=1-1/q=1-p/(2p-2)=(p-2)/(2p-2), so q'=(2p-2)/(p-2)
Assume $p>2$.  Note that Jensen's inequality implies that
\begin{equation}\label{three7}
\E X_{+}^{p/2}
=\E X_{+}^{\frac{p^{2}-2p}{2p-2}}
%***REMOVED expectation in "\E X_{+}^{\frac{p}{2p-2}}"
X_{+}^{\frac{p}{2p-2}}
\leq(\E X_{+}^{p})^{\frac{p-2}{2p-2}}
\cdot (\E X_{+})^{\frac{p}{2p-2}}
\stackrel{\eqref{three6}}{=}(\E X_{+}^{p})^{\frac{p-2}{2p-2}}
\cdot (\E\abs{X}/2)^{\frac{p}{2p-2}}.
\end{equation}
Also, by H\"{o}lder's inequality,
\begin{equation}\label{three7.1}
\E\abs{X}/2
\stackrel{\eqref{three6}}{=}\E X_{+}
=\E X_{+}1_{\{X>0\}}
\leq(\E X_{+}^{p/2})^{2/p}\cdot\P(X>0)^{\frac{p-2}{p}}.
\end{equation}

Applying \eqref{three7} to $X$ and $(-X)$ separately, exponentiating both sides to the power $(2p-2)/(p-2)$, and then adding the results,
%take both sides to the power (2p-2)/(p-2)
\begin{equation}\label{three50}
(\E X_{+}^{p/2})^{\frac{2p-2}{p-2}}
+(\E X_{-}^{p/2})^{\frac{2p-2}{p-2}}
\stackrel{\eqref{three7}}{\leq}2^{-\frac{p}{p-2}}(\E\abs{X})^{\frac{p}{p-2}}
(\E X_{+}^{p}+\E X_{-}^{p})
=2^{-\frac{p}{p-2}}(\E\abs{X})^{\frac{p}{p-2}}
\E\abs{X}^{p}.
\end{equation}
Applying \eqref{three7.1} to $X$ and $(-X)$ separately, exponentiating both sides to the power $-p/(p-2)$, and then adding the results,
\begin{equation}\label{three50.1}
(\E X_{+}^{p/2})^{-\frac{2}{p-2}}
+(\E X_{-}^{p/2})^{-\frac{2}{p-2}}
\stackrel{\eqref{three7.1}}{\leq}2^{\frac{p}{p-2}}[\P(X>0)+\P(X<0)]
%***REPLACED "\E\abs{X}^{p}" by
(\E\abs{X})^{-\frac{p}{p-2}}
\leq2^{\frac{p}{p-2}}(\E\abs{X})^{-\frac{p}{p-2}}.
\end{equation}

Finally, multiplying \eqref{three50} and \eqref{three50.1} gives \eqref{three5}, if $p>2$.

Assume $1<p<2$.  Then \eqref{three50}, \eqref{three50.1} and \eqref{three5} also hold.
%***REMOVED "for $1<p<2$."
To see this, we use the following two consequences of H\"{o}lder's inequality.
\begin{equation}\label{three9}
\E X_{+}^{p/2}
=\E X_{+}^{p/2}1_{\{X>0\}}
\leq(\E X_{+})^{p/2}\cdot\P(X>0)^{\frac{2-p}{2}}
\stackrel{\eqref{three6}}{=}(\E\abs{X}/2)^{p/2}\cdot\P(X>0)^{\frac{2-p}{2}}.
\end{equation}
\begin{equation}\label{three9.1}
\E\abs{X}/2
\stackrel{\eqref{three6}}{=}\E X_{+}
=\E X_{+}^{p-1} X_{+}^{2-p}
\leq(\E X_{+}^{p/2})^{\frac{2p-2}{p}}\cdot(\E X_{+}^{p})^{\frac{2-p}{p}}.
\end{equation}

Applying \eqref{three9} to $X$ and $(-X)$ separately, exponentiating both sides by the power $2/(2-p)$, and then adding the results,
%take both sides to the power (2p-2)/(p-2)
%***ADDED the next line
we obtain \eqref{three50.1},
%***\begin{equation}
%***REMOVED "\label{three51}" as it is neither used nor needed, and replaced {equation} with "$$"
$$
(\E X_{+}^{p/2})^{\frac{2}{2-p}}
+(\E X_{-}^{p/2})^{\frac{2}{2-p}}
\stackrel{\eqref{three9}}{\leq}2^{-\frac{p}{2-p}}(\E\abs{X})^{\frac{p}{2-p}}
\cdot[\P(X>0)+\P(X<0)]
\leq2^{-\frac{p}{2-p}}(\E\abs{X})^{\frac{p}{2-p}}.
%***\end{equation}
$$

Applying \eqref{three9.1} to $X$ and $(-X)$ separately, exponentiating both sides to the power $-p/(2-p)$, and then adding the results,
%take both sides to the power (2p-2)/(p-2)
%***ADDED the next line
we obtain \eqref{three50},
%***\begin{equation}
$$
%***REMOVED "\label{three51.1}" as it is neither used nor needed, and replaced {equation} with "$$"
%***CHANGED exponents in "(\E X_{+}^{p/2})^{\frac{2}{2-p}}+(\E X_{-}^{p/2})^{\frac{2}{2-p}}"
(\E X_{+}^{p/2})^{-\frac{2p-2}{2-p}}+(\E X_{-}^{p/2})^{-\frac{2p-2}{2-p}}
\stackrel{\eqref{three9.1}}{\leq}2^{\frac{p}{2-p}}(\E\abs{X})^{-\frac{p}{2-p}}(\E X_{+}^{p}+\E X_{-}^{p})
=2^{\frac{p}{2-p}}(\E\abs{X})^{-\frac{p}{2-p}}\E\abs{X}^{p}.
%***\end{equation}
$$
\end{proof}

\begin{lemma}\label{lemma11}
Let $p\in(1,\infty)\setminus\{2\}$.  For any $a,b>0$,
$$(a-b)^{2}\leq\frac{p^{2}-4p+4}{2p^{2}-4p+4}\cdot(a^{2}+b^{2}+a^{\frac{2}{2-p}}\cdot b^{\frac{2p-2}{p-2}}+b^{\frac{2}{2-p}}\cdot a^{\frac{2p-2}{p-2}}).$$
\end{lemma}
\begin{proof}
Without loss of generality, $a\geq b$.  Define $s\colonequals p/(p-2)$.  Since $\abs{s}>1$, for every $x\geq0$ we have
%e^x= 1+x+x^2/2+x^{3}/6+...
%e^-x=1-x+x^2/2-x^3/6+...
%so, e^x-e^-x is twice the odd order terms
\begin{equation}\label{three11}
e^{sx}+e^{-sx}
\geq e^{\abs{s}x}-e^{-\abs{s}x}
=2\sum_{j=0}^{\infty}\frac{\abs{s}^{2j+1}x^{2j+1}}{(2j+1)!}
\geq2\abs{s}\sum_{j=0}^{\infty}\frac{x^{2j+1}}{(2j+1)!}
= \abs{s}(e^{x}-e^{-x}).
\end{equation}
Set $x\colonequals(1/2)\log(a/b)$ and square the inequality \eqref{three11} to get
\begin{equation}\label{three12}
a^{s}b^{-s}+b^{s}a^{-s}+2\geq s^{2}(ab^{-1}+ba^{-1}-2).
\end{equation}
Multiplying both sides of \eqref{three12} by $ab$, then adding $(a-b)^{2}$ to both sides,
% a^{1+s}b^{1-s}+b^{1+s}a^{1-s}+2ab\geq s^{2}(a^{2}+b^{2}-2ab)=s^{2}(a-b)^{2}
% then add a^{2}+b^{2}-2ab=(a-b)^{2}
% blah\geq(1+s^{2})(a^{2}+b^{2})-2ab
\begin{equation}\label{three13}
a^{2}+b^{2}+a^{1-s}b^{1+s}+b^{1-s}a^{1+s}\geq(1+s^{2})(a-b)^{2}.
\end{equation}
And \eqref{three13} completes the lemma.
%1+s^{2}=1+p^{2}/(p-2)^{2}=[(p-2)^2+p^{2}]/(p-2)^2=[2p^{2}-4p+4]/[p^{2}-4p+4]
%1+s=1+p/(p-2)=(2p-2)/(p-2)
%1-s=1-p/(p-2)=-2/(p-2)=2/(2-p)
\end{proof}

\begin{lemma}\label{lemma12}
Let $p>1$ and let $X$ be a real random variable such that $\E\abs{X}^{p}<\infty$ and $\E X=0$.  Then
\begin{equation}\label{three15}
\big(\E X_{+}^{p/2}-\E X_{-}^{p/2}\big)^{2}
\leq\left(1-\frac{p^{2}}{2(p^{2}-2p+2)}\right)\cdot\E\abs{X}^{p}.
\end{equation}
\end{lemma}
\begin{proof}
If $p=2$ or if $X=0$ almost surely, then both sides are zero.   So, we may assume $p\in(1,\infty)\setminus\{2\}$ and $X$ is nonzero on a set of positive measure.  In this case, set $a\colonequals\E X_{+}^{p/2}$, $b\colonequals\E X_{-}^{p/2}$, and apply Lemma \ref{lemma11} and then \eqref{three5}.
\end{proof}

\begin{lemma}[Stroock-Varopoulos]\cite{stroock84,varo85}\label{lemma13}
Let $a,b\in\R$, $p>1$.  Then
\begin{equation}\label{three14}
(\phi_{p-1}(a)-\phi_{p-1}(b))(a-b)\geq\frac{4(p-1)}{p^{2}}(\phi_{p/2}(a)-\phi_{p/2}(b))^{2}.
\end{equation}
\end{lemma}
\begin{proof}
Applying the Cauchy-Schwarz inequality,
\begin{flalign*}
\frac{4}{p^{2}}(\abs{a}^{p/2}\sign(a)-\abs{b}^{p/2}\sign(b))^{2}
&=\left(\int_{a}^{b}\abs{t}^{(p/2)-1}dt\right)^{2}
\leq\int_{a}^{b}\abs{t}^{p-2}dt
\cdot\int_{a}^{b}dt\\
&=\frac{1}{p-1}(\abs{a}^{p-1}\sign(a)-\abs{b}^{p-1}\sign(b))(a-b).
\end{flalign*}
\end{proof}

%\int (f-\int f)^{2}\leq\int|\nabla f |^{2}
\begin{proof}[Proof of Theorem \ref{main2}]
Recall that for any $g,h:\Omega \to \R$ we have
\begin{equation}\label{ten10}
\E gLh=\frac{1}{2}\sum_{x,y \in \Omega} \left(-\E{\bf 1}_{\{x\}}L{\bf 1}_{\{y\}}\right) \cdot (g(x)-g(y))(h(x)-h(y)),
\end{equation}
and that for $x \neq y$ there is
\begin{equation}\label{ten11}
-\E{\bf 1}_{\{x\}}L{\bf 1}_{\{y\}}=\mu(\{y\}) \cdot \frac{\mathrm{d}}{\mathrm{d} t}(P_{t}{\bf 1}_{\{x\}})(y)\Big|_{t=0^{+}} \geq 0.
\end{equation}
Therefore,
\begin{flalign*}
\E \phi_{p-1}(f)Lf
&\stackrel{\eqref{ten10}\wedge\eqref{ten11}\wedge\eqref{three14}}{\geq}
\frac{4(p-1)}{p^{2}}\E\phi_{p/2}(f)L\phi_{p/2}(f)\\
&\stackrel{\eqref{ten1}}{\geq}\frac{4(p-1)}{Cp^{2}}\cdot\big(\E\phi_{p/2}(f)^{2}-(\E\phi_{p/2}(f))^{2}\big)\\
&=\frac{4(p-1)}{Cp^{2}}\left(\E\abs{f}^{p}-\big(\E f_{+}^{p/2}-\E f_{-}^{p/2}\big)^{2}\right)
\stackrel{\eqref{three15}}{\geq}
%***ADDED a missing factor in "\frac{4}{Cp^{2}}"
\frac{4(p-1)}{Cp^{2}}
\left(\frac{p^{2}}{2(p^{2}-2p+2)}\right)\E\abs{f}^{p}.
\end{flalign*}
%\begin{flalign*}
%\int_{\R^{n}}\abs{f}^{p-2}\vnormf{\nabla f}_{2}^{2}d\gamma_{n}
%&=\frac{1}{p-1}\int_{\R^{n}}\langle\nabla(\abs{f}^{p-1}\sign(f)),\nabla f\rangle d\gamma_{n}\\
%&\stackrel{\eqref{three14}}{\geq}\frac{4}{p^{2}}\int_{\R^{n}}\vnormf{\nabla(\abs{f}^{p/2}\sign(f))}_{2}^{2}d\gamma_{n}\\
%&\stackrel{\eqref{zero00}}{\geq}\frac{4}{p^{2}}\left(\int_{\R^{n}}[\abs{f}^{p/2}\sign(f)]^{2}d\gamma_{n}-(\int_{\R^{n}}\abs{f}^{p/2}\sign(f)d\gamma_{n})^{2}\right)\\
%&=\frac{4}{p^{2}}\left(\int_{\R^{n}}\abs{f}^{p}d\gamma_{n}-(\int_{\R^{n}}(f_{+}^{p/2}-f_{-}^{p/2})d\gamma_{n})^{2}\right)\\
%&\stackrel{\eqref{three15}}{\geq}\frac{4}{p^{2}}\left(\frac{p^{2}}{2(p^{2}-2p+2)}\right)\int_{\R^{n}}\abs{f}^{p}d\gamma_{n}.
%\end{flalign*}
%so, in the original inequatliy, we get 2/(p^{2}-2p+2).  for p=2, get 1!  as p\to\infty, get 1/p^{2}.  as p\to1, get
\end{proof}

\section{Heat Smoothing}\label{secheat}

%[fix signs]

We now show that Theorem \ref{main2} implies Theorem \ref{main}.

\begin{proof}[Proof of Theorem \ref{main}]
Note that $\E f=0$ implies $\E P_{t}f=0$ for all $t\geq0$.  So, by Theorem \ref{main2},
\begin{flalign*}
&\frac{\mathrm{d}}{\mathrm{d} t}\left(\exp\left(\frac{(2p-2)pt}{(p^{2}-2p+2)C}\right)\cdot \E|P_{t}f|^{p}\right)\\
&\quad=\exp\left(\frac{(2p-2)pt}{(p^{2}-2p+2)C}\right)\left( \frac{(2p-2)p}{(p^{2}-2p+2)C}\,\E|P_{t}f|^{p}-p\cdot
%***REPLACED \varphi by \phi in "\E\varphi_{p-1}"
\E\phi_{p-1}
(P_{t}f)LP_{t}f\right)\leq 0.
\end{flalign*}
\end{proof}

\begin{remark}\label{rk31}
One easily extends Theorem~\ref{main2} from real-valued functions to $f$ taking values in a Euclidean space, with the same constant. In particular,
we get the same statement for complex-valued functions. Indeed, it suffices to apply Theorem~\ref{main} to $f_{v}(x):=\langle f(x),v \rangle$ and average over $v$'s from the unit sphere. Note that $|\langle w,v \rangle|^{p}$ averaged over the unit sphere (with respect to the uniform measure) is proportional to $\| w\|^{p}$, and the proportionality constant will cancel out.
\end{remark}

\begin{remark}\label{rk33}
Let $\kappa(p)=\inf_{u>1} \frac{u^{\frac{p}{p-2}}+u^{-\frac{p}{p-2}}}{u-u^{-1}}$. Note that $\kappa(p)=\kappa(p')$ since
$\frac{p}{p-2}=\frac{pp'}{p-p'}$. A simple analysis of the proof shows that we may strengthen the assertion of Theorem~\ref{main2} to
\[
\| P_{t}f\|_{p} \leq \exp\left(-\frac{(4p-4)t}{Cp^{2}(1+\kappa(p)^{-2})}\right)\cdot \|f\|_{p}.
\]
We have established (in the proof of Lemma~\ref{lemma11}) the estimate $\kappa(p) \geq |\frac{p}{p-2}|$. One can do better, however. For example,
there is $\kappa(4)=\kappa(4/3)=2\sqrt{2}$ and $\kappa(6)=\kappa(6/5)=2$, so that for every mean-zero $f$ we have
\[
\| P_{t}f\|_{4} \leq e^{-\frac{2t}{3C}}\| f\|_{4},\,\,\, \| P_{t}f\|_{4/3} \leq e^{-\frac{2t}{3C}}\| f\|_{4/3},
\]
\[
\| P_{t}f\|_{6} \leq e^{-\frac{4t}{9C}}\| f\|_{6},\,\,\, \| P_{t}f\|_{6/5} \leq e^{-\frac{4t}{9C}}\| f\|_{6/5}.
\]
Also, one can easily strengthen the lower bound to $\kappa(p) \geq \sqrt{\frac{p^{2}+4p-4}{p^{2}-4p+4}}$. Indeed,
for $s=\frac{p}{p-2}$ we have $|s|>1$, so that $u^{|s|/2}-u^{-|s|/2} \geq |s|(u^{1/2}-u^{-1/2})$ for every $u>1$. Squaring this inequality,
we get $u^{s}+u^{-s} \geq 2+s^{2}(u+u^{-1}-2)$, so that
\[
\kappa(p)=\inf_{u>1} \frac{u^{s}+u^{-s}}{u-u^{-1}} \geq \inf_{u>1} \frac{2+s^{2}(u+u^{-1}-2)}{u-u^{-1}}=\sqrt{2s^{2}-1}=
\sqrt{\frac{p^{2}+4p-4}{p^{2}-4p+4}}
\]
which, for $p>1$ and mean-zero functions $f$, yields
\[
\| P_{t}f\|_{p} \leq e^{-\left( \frac{2}{pp'}+\frac{8}{p^{2}{p'}^{2}}\right)t/C}\| f\|_{p}.
\]
\end{remark}

\begin{corollary}\label{cor4}
Let $p \in (1,\infty) \setminus \{ 2\}$ and let $X$ be a mean-zero real random variable with $\E |X|^{p}<\infty$. Then
\[
\left(\kappa(p)^{2}+1\right) \cdot \left( \E X_{+}^{p/2} -\E X_{-}^{p/2}\right)^{2} \leq \E|X|^{p},
\]
where $\kappa(p)=\inf_{u>1} \frac{u^{\frac{p}{p-2}}+u^{-\frac{p}{p-2}}}{u-u^{-1}}$.  Moreover, the constant $\kappa(p)^{2}+1$ is optimal.
Furthermore, under the assumptions of Theorem \ref{main2},
\[
\E \phi_{p/2}(f)L\phi_{p/2}(f) \geq \frac{C^{-1}}{1+\kappa(p)^{-2}}\cdot \E|f|^{p},
\]
and the constant  $C^{-1}/\left(1+\kappa(p)^{-2}\right)$ is optimal.
\end{corollary}
\begin{proof}
The above bounds follow from an obvious strengthening of the proof of Theorem \ref{main2} as explained in Remark \ref{rk33}.  To check the optimality of the constants, choose $v>1$ such that
$v^{\frac{p}{p-2}}+v^{-\frac{p}{p-2}}=\kappa(p)(v-v^{-1})$ and set $\alpha=1/(1+v^{\frac{4}{p-2}})$ and $\beta=1-\alpha$, so that $\beta/\alpha=v^{\frac{4}{p-2}}$.
Then, for a mean-zero random variable $X$ such that $\P(X=\beta)=\alpha$ and $\P(X=-\alpha)=\beta$, we have
\[
\E |X|^{p}=(\alpha+\beta)\E|X|^{p}=(\alpha+\beta)(\alpha\beta^{p}+\beta\alpha^{p})=\alpha^{2}\beta^{p}+\beta^{2}\alpha^{p}+\alpha\beta^{p+1}+\beta\alpha^{p+1},
\]
and thus
\begin{flalign*}
(\alpha\beta)^{-\frac{p+2}{2}}\E|X|^{p}
&=(\alpha/\beta)^{\frac{p-2}{2}}+(\beta/\alpha)^{\frac{p-2}{2}}+(\alpha/\beta)^{p/2}+(\beta/\alpha)^{p/2}\\
&=v^{2}+v^{-2}+v^{\frac{2p}{p-2}}+v^{-\frac{2p}{p-2}}=\left( v^{\frac{p}{p-2}}+v^{-\frac{p}{p-2}}\right)^{2}+\left( v-v^{-1}\right)^{2}\\
&=\left(\kappa(p)^{2}+1\right)\left(v-v^{-1}\right)^{2}=\left(\kappa(p)^{2}+1\right)
\left( (\alpha/\beta)^{\frac{p-2}{4}}-(\beta/\alpha)^{\frac{p-2}{4}}\right)^{2}\\
&=(\alpha\beta)^{-\frac{p+2}{2}}\left(\kappa(p)^{2}+1\right)\left( \alpha\beta^{p/2}-\beta\alpha^{p/2}\right)^{2}\\
&=(\alpha\beta)^{-\frac{p+2}{2}}\left(\kappa(p)^{2}+1\right)\left( \E X_{+}^{p/2} -\E X_{-}^{p/2} \right)^{2}.
\end{flalign*}
To see that the constant $C^{-1}/\left(1+\kappa(p)^{-2}\right)$ in the second claim is also optimal, consider $\Omega=\{-\alpha,\beta\}$ with
$\mu=\alpha\delta_{\beta}+\beta\delta_{-\alpha}$, $L=C^{-1} \cdot (Id-\E)$, and $f(x)=x$.  With this definition of $L$, equality holds in \eqref{ten1} for any function.  So, arguing as in the proof of Theorem \ref{main2}, we have $\E\phi_{p/2}(f)L\phi_{p/2}(f)=\E|f|^{p}-(\E f_{+}^{p/2}-\E f_{-}^{p/2})^2$, and we then use the equality that holds from the first part of the present corollary.
\end{proof}
\begin{remark}

We now show that the dependence on $p$ of Theorem \ref{main2} is of optimal order as $p\to1$ or as $p\to\infty$.  In what follows, $L$ is the generator of the standard one-dimensional Ornstein-Uhlenbeck semigroup and $\gamma$ is the standard ${\mathcal{N}}(0,1)$
Gaussian measure.

For $\eps>0$, let $g_{\eps}:\R \rightarrow \R$ be an increasing $2$-Lipschitz function, smooth on $\R \setminus \{ 0\}$ and such that
$g_{\eps}(x)=x$ for $|x|>\eps$ and $g_{\eps}=\phi_{3p/2}$ on some neighbourhood of zero.
%{\bf Comment: the existence of such a function is an easy calculus exercise, I guess we do not need to prove it in our paper}
Furthermore, let $f_{\eps}=\phi_{2/p} \circ g_{\eps}$. Then $f_{\eps}$ is a smooth function
and it belongs to the domain of $L$. We have
\[
\E_{\gamma}|f_{\eps}|^{p}=\E_{\gamma}g_{\eps}^{2} \geq \int_{\R \setminus [-\eps,\eps]} x^{2}\,\mathrm{d} \gamma(x) \stackrel{\eps \to 0^{+}}{\longrightarrow} 1
\]
and
\[
\frac{p^{2}}{4(p-1)}\cdot \E_{\gamma} \phi_{p-1}(f_{\eps})Lf_{\eps}=
\frac{p^{2}}{4(p-1)}\int_{\R} (\phi_{p-1} \circ f_{\eps})'f_{\eps}'\,\mathrm{d} \gamma=
\]
\[
\int_{\R} \left( \left( \phi_{p/2} \circ f_{\eps}\right)'\right)^{2} \mathrm{d} \gamma=
\int_{\R} (g_{\eps}')^{2} \mathrm{d} \gamma \leq \gamma(\R \setminus [-\eps,\eps])+4\gamma([-\eps,\eps]) \stackrel{\eps \to 0^{+}}{\longrightarrow} 1.
\]

\end{remark}

\begin{remark} \label{reminf}
%Infinite Spaces}\label{secinf}
The proof of Theorem \ref{main} used the finiteness of the space $\Omega$ in a nonessential way.  Infinite spaces require a bit more care, so we have chosen the above presentation.  However, the reader can find a proof of Theorem \ref{main} in Section \ref{secsemi} which applies to infinite spaces, and which only uses the semigroup itself.
%*** REMOVED reference to grigoryan paper, since it is in some obscure proceedings
%The proof of Theorem \ref{main} used the finiteness of the space $\Omega$ in a nonessential way.  Infinite spaces require a bit more care, so we have chosen the above presentation.
%The only place where finiteness was used was in \eqref{ten10}. A suitable replacement for this inequality appears e.g. in \cite[Eq. (2.7)]{ledoux99},
%where it is shown that for all $t>0$, there exist functions $p_{t}\colon\Omega\times\Omega\to[0,\infty]$ such that, for any
%%*** REPLACED "smooth" by
%regular enough
%%***BECAUSE it is too unclear what could be meant by "smooth" in this particular context
%functions $f,g\colon\Omega\to\R$, we have
%
%\begin{equation}\label{zero3}
%\E fLg=\lim_{t\to0^{+}}\frac{1}{2t}\int_{\Omega}\int_{\Omega}(f(x)-f(y))(g(x)-g(y))p_{t}(x,y)d\mu(x)d\mu(y).
%\end{equation}
\end{remark}

\medskip \noindent
To conclude the section, we prove Theorem~\ref{thm13}.
\begin{proof}[Proof of Theorem \ref{thm13}]
% \theta/2+(1-\theta)/p'=1/p.
%just let p'=2p, get
% \theta/2+(1-\theta)/(2p)=1/p, \theta/2-\theta/(2p)=1/p-(1/2)(1/p)=1/(2p), (1/2)\theta(1-1/p)=1/(2p), \theta(1-1/p)=1/p, \theta(p-1)=1, \theta=1/(p-1)
%1-\theta=(p-1)/(p-1)-1/(p-1)=(p-2)/(p-1)
% p(p-2)+p=p(p-1)
Recall that
\begin{equation}\label{six1}
\forall\,1\leq q\leq\infty,\quad
\vnormf{P_{t}f}_{q}\leq\vnorm{f}_{q}.
\end{equation}
Also, if $\E f W_{S}=0$ for all $S\subset\{1,\ldots,n\}$ with $\abs{S}<k$, then for all $t>0$,
\begin{equation}\label{six2}
\vnormf{P_{t}f}_{2}\leq e^{-tk}\vnorm{f}_{2}.
\end{equation}

Let $f\colon\{-1,1\}^{n}\to\{-1,0,1\}$. For such $f$ we have that $\E[|f|^p] = \E[|f|]$ for all $p$
%***ADDED the next line
and $\| P_{t}f\|_{\infty} \leq 1$ for all $t \geq 0$.
Now,
%***REPLACED "If" by
if
$p>2$, then

%***REMOVED (by \ignore) the following lines
\ignore{
H\"{o}lder's inequality says
\begin{flalign*}
\E\absf{P_{t}f}^{p}
&\leq(\E \abs{P_{t}f}^{2})^{\frac{p}{2(p-1)}}(\E\abs{P_{t}f}^{2p})^{\frac{p-2}{2(p-1)}}\\
&\stackrel{\eqref{six1}\wedge\eqref{six2}}{\leq} e^{-tkp/(p-1)}(\E\abs{f}^{2})^{\frac{p}{2(p-1)}}
(\E\abs{f}^{2p})^{\frac{p-2}{2(p-1)}}
= e^{-tkp/(p-1)}\E\abs{f}^{p}.
\end{flalign*}
}
%***AND REPLACED them by a simpler and stronger estimate (next line)
$$ \E|P_{t}f|^{p} \leq \E|P_{t}f|^{2} \leq e^{-2tk}\E f^{2}=e^{-2tk}\E|f|^{p}.$$
If $1<p<2$, then from H\"{o}lder's inequality,
%***REMOVED "again" and REMOVED unnecessary brackets around exponent p-1 in the formula below
$$
\E\abs{P_{t}f}^{p}
\leq(\E\abs{P_{t}f})^{2-p}(\E\abs{P_{t}f}^{2})^{p-1}
\stackrel{\eqref{six1}\wedge\eqref{six2}}{\leq}(\E\abs{f})^{2-p}e^{-tk2(p-1)}
(\E\abs{f}^{2})^{p-1}
=e^{-2tk(p-1)}\E\abs{f}^{p}.
$$
\end{proof}

\section{An Interpolation Proof of Heat Smoothing}\label{secnaz}

After the results of this paper were presented at some seminars, Fedor Nazarov informed the authors about an alternative complex interpolation proof of the strong contractivity bounds. With his kind permission, we present a simple version of such an approach, strongly inspired by Nazarov's proof but different from it. Nazarov proved that a perturbation of a Markov operator $P$ by an appropriately chosen small multiplicity of the expectation, i.e., $T=P-\delta \cdot \E$, is a contraction (a trick he ascribed to Bernstein), and his interpolation bounds were more elaborate than ours.

The $L^p$ spaces considered in this section are complex.

\begin{proposition}
\label{RTNazarov}
Let $(\Omega, \mu)$ be a probability space (for brevity, we leave simple measurability considerations to the reader) and let $P: \lj \to \lj$
be a linear operator with $P\1=\1$, $\| P\|_{\lj \to \lj} \leq 1$, and $\| P\|_{\li \to \li} \leq 1$. Also, we assume that $P$ is
mean preserving, i.e., $\E Pf=\E f$ for every $f \in \lj$. In particular, these conditions are satisfied by every, not necessarily symmetric, Markov operator $P$ for which $\mu$ is an invariant measure. Furthermore, let as assume that there exists $\eps \in [0,1)$ such that
$
\| Pg\|_{\ld}^{2} \leq (1-\eps)\| g\|_{\ld}^{2}
$
for every mean-zero $g$. Then, for $p \in (1,\infty)$,
\[
\| Pg\|_{\lp} \leq \left( 1-2^{2-p^{*}}\eps\right)^{1/p^{*}} \cdot \| g\|_{\lp}
\]
for every mean-zero $g$, where $p^{*}=\max\left( p, \frac{p}{p-1}\right)$.
\end{proposition}

\begin{proof}
Let $c$ be a positive constant, to be specified later. We define a measure $\tmu$ on a new space $\tOmega=\Omega \times \{ 0,1\}$ by
setting $\tmu(A \times \{0\})=\mu(A)$ and $\tmu(A \times \{ 1\})=c^{2}\eps \cdot \mu(A)$, and a linear operator $T: \lj \to \ltj$ by
$(Tf)(\omega,0)=c \cdot (Pf)(\omega)$ and $(Tf)(\omega,1)=f(\omega)-\E f$ for $\omega \in \Omega$. Clearly,
$\| T\|_{\lj \to \ltj} \leq c+2c^{2}\eps$ and $\| T\|_{\li \to \lti} \leq \max(c,2)$.

For $f \in \ld$, let $g=f-\E f$, so that $\E g=0$.
Therefore also $\E Pg=0$ and thus
\[
\| Pf\|_{\ld}^{2}=\| Pg+\E f\|_{\ld}^{2}=\| Pg\|_{\ld}^{2}+|\E f|^{2}
\]
\[
\leq (1-\eps)\| g\|_{\ld}^{2}+|\E f|^{2}=
\| g+\E f\|_{\ld}^{2}-\eps\| g\|_{\ld}^{2}=\| f\|_{\ld}^{2}-\eps\| g\|_{\ld}^{2},
\]
so that
\[
\| Tf\|_{\ltd}^{2}=c^{2}\| Pf\|_{\ld}^{2}+c^{2}\eps\| g\|_{\ld}^{2} \leq c^{2}\| f\|_{\ld}^{2}.
\]
We have proved that $\| T\|_{\ld \to \ltd} \leq c$.

\bigskip

$\bullet$ For $p \in (1,2]$, by the Riesz-Thorin theorem,
$$
\| T\|_{\lp \to \ltp}
\leq \| T\|_{\lj \to \ltj}^{(2-p)/p} \| T\|_{\ld \to \ltd}^{(2p-2)/p} 
\leq(c+2c^{2}\eps)^{\frac{2-p}{p}} \cdot c^{\frac{2p-2}{p}}=c(1+2c\eps)^{\frac{2-p}{p}}.
$$
Thus for any $g \in \lp$ with $\E g=0$,
\[
c^{p}\| Pg\|_{\lp}^{p}+c^{2}\eps\| g\|_{\lp}^{p}=\| Tg\|_{\ltp}^{p} \leq c^{p}(1+2c\eps)^{2-p}\| g\|_{\lp}^{p}.
\]
Note that for $c=\left( 2^{\frac{1}{p-1}}-2\eps\right)^{-1}$, we have
\[
(1 + 2 c \eps)^{2-p}-c^{2-p} \eps = c^{2-p}((c^{-1} + 2 \eps)^{2-p} - \eps) = \frac{c^{1-p}}{2},
\]
and therefore
\[
\| Pg\|_{\lp} \leq \left( (1+2c\eps)^{2-p}-c^{2-p}\eps\right)^{1/p}\| g\|_{\lp}=
%\left( 1-2^{-\frac{2-p}{p-1}}\eps\right)^{p-1}\| g\|_{\lp}^{p}=
\left(1-2^{2-p^{*}}\eps\right)^{1/p^{*}}\| g\|_{\lp}.
\]

\bigskip

$\bullet$ For $p \in [2,\infty)$, we simply set $c=2$. Then we have $\| T\|_{\ld \to \ltd} \leq 2$ and $\| T\|_{\li \to \lti} \leq 2$,
and thus, by the Riesz-Thorin theorem, also $\| T\|_{\lp \to \ltp} \leq 2$. Therefore, for any mean-zero $g \in \lp$,
\[
2^{p}\| Pg\|_{\lp}^{p}+4\eps\| g\|_{\lp}^{p}=\| Tg\|_{\ltp}^{p} \leq 2^{p}\| g\|_{\lp}^{p},
\]
so that $\| Pg\|_{\lp} \leq \left( 1-2^{2-p}\eps\right)^{1/p}\|g\|_{\lp}=\left( 1-2^{2-p^{*}}\eps\right)^{1/p^{*}}\|g\|_{\lp}$.
\end{proof}

\begin{remark}
Obviously, if $P: \ld \to \ld$ is a Markov operator such that $\| Pg\|_{\ld}^{2} \leq (1-\eps)\| g\|_{\ld}^{2}$ for every real-valued mean-zero function $g$ then we have also $\| Pg\|_{\ld}^{2} \leq (1-\eps)\| g\|_{\ld}^{2}$ for every complex-valued mean-zero function $g$.
\end{remark}

\section{A Semigroup Proof of Heat Smoothing}\label{secsemi}
We were asked by experts in operator if the results of Section \ref{secpoincare} and Section \ref{secheat} can be extended to the setting of infinite probability spaces.
The proof in this setting, which does not use the notion of the semigroup generator, is outlined in the following section.

%Several experts in operator theory raised concerns whether the reasoning of Sections \ref{secpoincare} and \ref{secheat} can be extended to infinite probability space
%setting. Therefore we decided to outline a proof of heat smoothing which does not use the notion of the semigroup generator.

We will need some standard and simple bounds.
\begin{lemma}
\label{x1}
Let $(P_{s})_{s \geq 0}$ be a semigroup of symmetric linear contractions on some inner product space $(\cH, \| \cdot \|)$. Then for any $\eps, t>0$ we have
$\| P_{\eps+t}-P_{\eps}\|_{\cH \to \cH} \leq 2t/\eps$.

\end{lemma}

\begin{proof}
Indeed, for $f \in \cH$, let $h=P_{t}f-f$, so that $P_{\eps+t}f-P_{\eps}f=P_{\eps}h$. Then
\begin{flalign*}
&(\eps/t)^{2}\| P_{\eps+t}f-P_{\eps}f\|^{2} \leq
\sum_{k=0}^{[\eps/t]}\sum_{l=0}^{[\eps/t]} \| P_{\eps-\frac{k+l}{2}t}P_{\frac{k+l}{2}t}h\|^{2}
\leq \sum_{k=0}^{[\eps/t]}\sum_{l=0}^{[\eps/t]} \| P_{\frac{k+l}{2}t}h\|^{2}
=\sum_{k=0}^{[\eps/t]}\sum_{l=0}^{[\eps/t]} \langle P_{kt}h,P_{lt}h \rangle\\
&\quad=\left\langle \sum_{k=0}^{[\eps/t]} \left(P_{(k+1)t}f-P_{kt}f\right), \sum_{l=0}^{[\eps/t]} \left(P_{(l+1)t}f-P_{lt}f\right) \right\rangle
=\left\| P_{t[\eps/t]+t}f-f\right\|^{2} \leq (2\|f\|)^{2}.
\end{flalign*}
\end{proof}

The following is a weak version of the Stroock-Varopoulos inequality.
\begin{lemma}
\label{x2}
Let $P$ be a symmetric Markov operator on a probability space $(\Omega,\mu)$. Then, for any $f \in \li$ and $p \in (1,\infty)$,
\[
(p-2)^{2}\,\E|f|^{p}+4(p-1)\E \phi_{p/2}(f)P\left(\phi_{p/2}(f)\right) \geq p^{2}\,\E \phi_{p-1}(f)Pf.
\]
\end{lemma}

\begin{proof}
Upon simple algebraic transformations, Lemma \ref{lemma13} implies that, $\mu$-a.e.,
\[
\forall_{a \in \Q}\,\,\,
(p-2)^{2}(|a|^{p}+|f|^{p})+8(p-1)\phi_{p/2}(a)\phi_{p/2}(f) \geq p^{2}\left(a\phi_{p-1}(f)+\phi_{p-1}(a)f\right).
\]
Since $P$ is linear and positivity preserving, $\mu$-a.e. we have also
$$
\forall_{a \in \Q}\,\,\,
(p-2)^{2}\left(|a|^{p}+P(|f|^{p})\right)+8(p-1)\phi_{p/2}(a)P\left(\phi_{p/2}(f)\right)\geq p^{2}\left(aP\left(\phi_{p-1}(f)\right)+\phi_{p-1}(a)Pf\right).
$$
By the continuity in $a$, the same holds true $\mu$-a.e. when $\forall_{a \in \Q}$ is replaced by $\forall_{a \in \R}$.
In particular, $\mu$-a.e.,
$$
(p-2)^{2}\left(|f|^{p}+P(|f|^{p})\right)+8(p-1)\phi_{p/2}(f)P\left(\phi_{p/2}(f)\right)
\geq p^{2}\left(fP\left(\phi_{p-1}(f)\right)+\phi_{p-1}(f)Pf\right).
$$
We finish by taking the expectation of both sides and using the symmetry of $P$ (together with the fact that symmetric Markov operators
are mean-preserving).
\end{proof}

\begin{definition} \label{x4}
Let $(P_{t})_{t \geq 0}$ be a symmetric Markov semigroup on a probability space $(\Omega,\mu)$. For $M, \eps>0$ we will denote by $\cC(M,\eps)$
the class of all functions of the form $P_{\eps}\vp$, where $\vp \in L^{\infty}(\Omega,\mu)$ and $\| \vp\|_{\infty} \leq M$.
%Note that if $f \in \cC(M,\eps)$ then $P_{t}f \in \cC(M,\eps)$ for every $t>0$.
Furthermore, for $p \in (1,\infty)$ we set
\[
\alpha_{p,M,\eps}(t)=\sup_{f \in \cC(M,\eps)} \left( (p-1)\E|f|^{p}+\E|P_{t}f|^{p}-p\E\phi_{p-1}(f)P_{t}f\right).
\]
\end{definition}

\begin{lemma} \label{x5}
For any $p \in (1,\infty)$ and $M,\eps>0$, we have $\alpha_{p,M,\eps}(t)/t \to 0$ as $t \to 0^{+}$.
\end{lemma}

\begin{proof}
Let $f \in \cC(M,\eps)$, i.e., $f=P_{\eps}\vp$ and $\|\vp\|_{\infty} \leq M$.
Let $\rho_{p}=\sup_{s \neq 1} \frac{|s|^{p}-ps+p-1}{(|s|+1)^{p-2}(s-1)^{2}}$ for $p>2$, while for $p \in (1,2]$ let
$\rho_{p}=\sup_{s \neq 1}\frac{|s|^{p}-ps+p-1}{|s-1|^{p}}$. Note that $\rho_{p}<\infty$. By the homogeneity,
$$
(p-1)\E|f|^{p}+\E|P_{t}f|^{p}-p\E\phi_{p-1}(f)P_{t}f
\leq \rho_{p}\,\E(|P_{t}f|+|f|)^{p-2}(P_{t}f-f)^{2} 
\leq (2M)^{p-2}\rho_{p}\cdot \E(P_{t}f-f)^{2},
$$
for $p>2$, and for $p \in (1,2]$ we have
\[
(p-1)\E|f|^{p}+\E|P_{t}f|^{p}-p\E\phi_{p-1}(f)P_{t}f \leq \rho_{p}\,\E|P_{t}f-f|^{p} \leq \rho_{p}\cdot \left( \E(P_{t}f-f)^{2}\right)^{p/2}.
\]
We finish by Lemma~\ref{x1}: $\E(P_{t}f-f)^{2}=\E(P_{t+\eps}\vp-P_{\eps}\vp)^{2} \leq (2t/\eps)^{2}\E\vp^{2} \leq (2M/\eps)^{2}t^{2}$.
\end{proof}

Now we are in a position to recover Remark \ref{rk33} by a purely semigroup approach:

\begin{proposition} \label{x6}

Let $(P_{t})_{t \geq 0}$ be a symmetric Markov semigroup on a probability space $(\Omega, \mu)$. We also assume that it is a $C_{0}$-semigroup
on $L^{q}(\Omega,\mu)$ for some $q \in [1,\infty)$, i.e., $\| P_{t}g-g\|_{q} \to 0$ as $t \to 0^{+}$ for every $g \in L^{q}(\Omega, \mu)$.
Furthermore, let us assume that there exists a positive constant $C$ such that $\| P_{t}f\|_{2} \leq e^{-t/C}\|f\|_{2}$ for every mean-zero function
$f \in L^{2}(\Omega,\mu)$ and every $t>0$. Then, for every $p \in (1,\infty)$, $t>0$, and any mean-zero $f$,
\begin{equation} \label{semigr}
\| P_{t}f\|_{p} \leq \exp\left(-\frac{(4p-4)t}{Cp^{2}(1+\kappa(p)^{-2})}\right) \cdot \| f\|_{p}.
\end{equation}

\begin{proof}
Let $\cC_{0}(M,\eps)=\{ f \in \cC(M,\eps): \E f=0 \}$ and let $L^{p}_{0}(\Omega,\mu)$ be the subspace of mean-zero functions in $L^{p}(\Omega,\mu)$.
The $C_{0}$-semigroup condition implies that $\bigcup_{M,\eps>0} \cC_{0}(M,\eps)$ is dense in $L^{p}_{0}(\Omega,\mu)$. Indeed, for
$\vp \in L^{\infty}_{0}(\Omega,\mu)$ by assumption we have $P_{\eps}\vp \stackrel{\eps \to 0^{+}}{\to} \vp$ in $L^{q}(\Omega,\mu)$, and thus also
in $L^{p}(\Omega, \mu)$, since $\E|P_{\eps}\vp-\vp|^{p} \leq (2\| \vp\|_{\infty})^{p-q}\E|P_{\eps}\vp-\vp|^{q}$ if $p>q$ and
$\| P_{\eps}\vp-\vp\|_{p} \leq \| P_{\eps}\vp-\vp\|_{q}$ if $p \leq q$. Now it suffices to note that
bounded mean-zero functions are dense in $L_{0}^{p}(\Omega,\mu)$ and $P_{\eps}\vp \in \cC_{0}(\| \vp\|_{\infty}, \eps)$ .

Since, as a contraction, $P_{t}$ is uniformly continuous on $L^{p}(\Omega,\mu)$, it is enough to prove the assertion for $f \in \cC_{0}(M, \eps)$ for every $M, \eps>0$.
By assumption,
\[
\E\phi_{p/2}(f)P_{t}\left(\phi_{p/2}(f)\right)-\left(\E \phi_{p/2}(f)\right)^{2}=
\left\| P_{t/2}\left(\phi_{p/2}(f)
-\E\phi_{p/2}(f)\right)\right\|_{2}^{2}
\]
\[
\leq
e^{-t/C}\left\| \phi_{p/2}(f)
-\E\phi_{p/2}(f)\right\|_{2}^{2}=
e^{-t/C} \cdot \left( \E|f|^{p}-\left(\E \phi_{p/2}(f)\right)^{2}\right).
\]

By Corollary \ref{cor4}, $\left( \E\phi_{p/2}(f)\right)^{2} \leq
\E|f|^{p}/\left( \kappa(p)^{2}+1\right)$. These inequalities, together with Lemma~\ref{x2} and Definition~\ref{x4}, yield
\[
\E|P_{t}f|^{p} \leq
\left( 1-\frac{(4p-4)(1-e^{-t/C})}{(1+\kappa(p)^{-2})p}\right)
\cdot \E|f|^{p}+\alpha_{p,M,\eps}(t).
\]
Thus, for positive integers $k$ and $n$, by a simple induction on
$k$, we have
\[
\E|P_{kt/n}f|^{p} \leq
\left( 1-\frac{(4p-4)(1-e^{-(t/n)/C})}{(1+\kappa(p)^{-2})p}\right)^{k}
\cdot \E|f|^{p}+k\alpha_{p,M,\eps}(t/n)
\]
--~it suffices to consider $t/n$ instead of $t$ and note that $f \in \cC_{0}(M,\eps)$ implies
$P_{s}f \in \cC_{0}(M,\eps)$ for all $s \geq 0$. Taking $k=n$ and $n \to \infty$ ends the proof since, by Lemma~\ref{x5},
$n\alpha_{p,M,\eps}(t/n) \to 0$.

\end{proof}

\begin{remark} \label{x7}
Some authors include the $C_{0}$-semigroup assumption for $q=1$ into the very definition of Markov semigroups. It is easy to prove that
a Markov semigroup is a $C_{0}$-semigroup on $L^{q}(\Omega,\mu)$ for some $q \in [1,\infty)$ if and only if it  is a $C_{0}$-semigroup on
$L^{q}(\Omega,\mu)$ for every $q \in [1,\infty)$. In many cases it is convenient to test the property for $q=2$. In particular, if
$(P_{t})_{t \geq 0}$ can be represented, by means of functional calculus, as $e^{-tL}$ for some positive semidefinite self-adjoint operator
$L$ on $L^{2}(\Omega,\mu)$, this property (for $q=2$) easily follows from the spectral theorem.
\end{remark}

For $f$ belonging to the domain of the generator $L$, we obtain the $L^{p}$ Poincar\'e inequality simply by differentiating (\ref{semigr}) at zero.

\end{proposition}

\section{Talagrand's Inequality for Tail Space}\label{sectal}

\begin{proof}[Proof of Theorem \ref{thm4}]

%[just use theta=1/2]
The argument follows the one in \cite{ledoux12}.  Let $f\colon\{-1,1\}^{n}\to\R$ with $\E f W_{S}=0$ for all $S\subset\{1,\ldots,n\}$ with $\abs{S}<k$.
Hence $\| P_{1/k}f\|_{2} \leq e^{-1}\| f\|_{2}$ and thus
\begin{equation} \label{intbound}
(1-e^{-2})\E f^{2} \leq \E f^{2} - \E(P_{1/k}f)^{2}=-\int_{0}^{1/k}\frac{d}{dt}\E(P_{t})^{2}\,dt=2\int_{0}^{1/k}\E P_{t}fLP_{t}f\,dt
\end{equation}
$$
=2\int_{0}^{1/k} \sum_{i=1}^{n} \E(D_{i}P_{t}f)^{2}\,dt=2\sum_{i=1}^{n}\int_{0}^{1/k} \E(P_{t}D_{i}f)^{2}\,dt \leq
2\sum_{i=1}^{n} \int_{0}^{1/k} \| D_{i}f\|_{1+e^{-2t}}^{2}\,dt,
$$
where the last inequality is the usual hypercontractive bound \cite{bonami70,nelson73,gross75}.

By H\"older's inequality, for $0<q<p<2$ we have
\begin{equation}\label{holdon}
\E |g|^{p}=\E |g|^{\frac{(2-p)q}{2-q}}|g|^{\frac{2(p-q)}{2-q}}
\leq \left(\E|g|^{q}\right)^{\frac{2-p}{2-q}}\left(\E g^{2}\right)^{\frac{p-q}{2-q}}.
\end{equation}
Applying this estimate to $g=D_{i}f$,
$q=1+e^{-2/k}$, and $p=1+e^{-2t}$ with $t \in (0,1/k)$,
$$
\| D_{i}f\|_{1+e^{-2t}}^{2}
\leq\| D_{i}f\|_{2}^{2} \left( \| D_{i}f\|_{1+e^{-2/k}}/\| D_{i}f\|_{2}\right)^{\frac{2\tanh t}{\tanh (1/k)}}
\leq \| D_{i}f\|_{2}^{2} \left( \| D_{i}f\|_{1+e^{-2/k}}/\| D_{i}f\|_{2}  \right)^{2tk}
$$
since $t \mapsto \frac{\tanh t}{t}$ is decreasing on $(0,\infty)$. Therefore
\begin{flalign*}
\int_{0}^{1/k} \| D_{i}f\|_{1+e^{-2t}}^{2}\,dt
&\leq \| D_{i}f\|_{2}^{2} \int_{0}^{1/k} \left( \| D_{i}f\|_{1+e^{-2/k}}/\| D_{i}f\|_{2}  \right)^{2tk}\,dt\\
&=\| D_{i}f\|_{2}^{2} \frac{1- \left( \| D_{i}f\|_{1+e^{-2/k}}/\| D_{i}f\|_{2}  \right)^{2}}{2k \log  \left( \| D_{i}f\|_{2}/\| D_{i}f\|_{1+e^{-2/k}}  \right)} \\
&\leq\frac{1}{k}\| D_{i}f\|_{2}^{2}\min\left(1, \frac{1}{2\log \left( \| D_{i}f\|_{2}/\| D_{i}f\|_{1+e^{-2/k}}  \right)}\right),
\end{flalign*}
where we have used the fact that $1-a^{-2} \leq 2\log a$ for $a \geq 1$.
Together with \eqref{intbound} this ends the proof of the first inequality of Theorem \ref{thm4}.

Applying \eqref{holdon} to $g=D_{i}f$, $q=1$, and $p=1+e^{-2/k}$, we get  $\| D_{i}f \|_{2}/\| D_{i}f\|_{1+e^{-2/k}} $\\$\geq
\left ( \| D_{i}f\|_{2}/\| D_{i}f\|_{1}\right)^{\tanh(1/k)}$. Since $k\tanh(1/k) \geq \tanh(1)$, the second inequality of Theorem \ref{thm4}
easily follows.
\end{proof}

\section{The Coding Tribes Function}\label{sechatami}

%***Fixed a few English errors below, also changed 'balanced' to 'mean zero,' to be consistent with the rest of the article

Recall that in Proposition~\ref{prop1} Ben-Or and Linial constructed a Boolean function with mean zero and all of whose influences are $O(\log n/n)$. The results of KKL in Theorem~\ref{kkl} imply that it is impossible for the maximum influence of a mean zero Boolean function to be of lower order.  In Question~\ref{q1} Hatami and Kalai asked if the KKL result can be strengthened if the function $f$ satisfies additionally that $\E[f W_S] = 0$ for all $S$ with $|S| < k$ where $k(n) \to \infty$ as
$n \to \infty$.

The KKL result in fact
%***CHANGED "imply" to
implies
that mean zero Boolean functions which are invariant under permutation of the inputs have an influence sum which is
$\Omega(\log n)$.
We first note that, by taking the Ben-Or and Linial tribes function $f$ and letting $g(x_1,\ldots,x_n,y_1,\ldots,y_k) = f(x) y_1 \ldots y_k$,  we obtain a function all of whose Fourier coefficients up to level
$k$ vanish and such that its sum of influences is $O(\log n + k)$.
Thus one cannot improve on the KKL sum of influence result unless $k/\log n \to \infty$ as $n\to\infty$. In this section we will construct an example of a function all of whose Fourier coefficients up to level $\Omega(\log n)$ vanish and all of whose individual influences are at most $O(\log n/n)$ thus proving Theorem~\ref{thm30} and answering in the negative Question~\ref{q1}.
%*** Change 'coefficients' to 'fourier coefficients' above
%*** added 'as n\to\infty'

We denote by $L^{>k}(\{-1,1\}^{n})$ the space of all functions $f\colon\{-1,1\}^{n}\to\R$ such that $\E fW_{S}=0$ for all $S\subset\{1,\ldots,n\}$ with $\abs{S}\leq k$.  We denote by $L_+^{>k}(\{-1,1\}^{n})$ the space of all functions $f\colon\{-1,1\}^{n}\to\{-1,1\}$ such that $\E fW_{S}=0$ for all $S\subset\{1,\ldots,n\}$ with $1\leq\abs{S}\leq k$.  The difference between the two families is that the latter functions are allowed to have non-zero expectation.

%*** changed iff to if and only if
We will use the convention that $1$ and $-1$ map to the logical values TRUE and FALSE, respectively.
Thus for $x_{1},\ldots,x_{n}\in\{-1,1\}$, we have $x_1 \vee \cdots \vee x_n  = -1$ if and only if $x_1 = \cdots = x_n = -1$, and
$x_1 \wedge \cdots \wedge x_n  = 1$ if and only if $x_1= \cdots = x_n = 1$.

Our strategy is to construct a function in $L_+^{>k}(\{-1,1\}^{n})$ with low influences and small mean and then ``correct'' it so that it has mean zero

The basic idea behind the construction is the following: we want to mimic the construction of the tribes function. Recall that the tribe function is given by
\[
(x_1 \wedge \ldots \wedge x_r) \vee \ldots \vee (x_{(b-1)r+1} \wedge \ldots \wedge x_{br})
\]
In our construction, which we call the {\em Coding Tribes } function instead of substituting AND functions into the arguments of an OR function, we will substitute functions in $L_+^{>k}$ into the arguments of an OR function.

For example for $k=1$, instead of the AND function on $r$ bits we will take the function ALLEQ on $r+1$ bits, where
$\mathrm{ALLEQ}(x_1,\ldots,x_{r+1})$ takes the value $1$ exactly
%***ADDED "if"
if
the $x_i$ are all $1$ or all $-1$. Clearly the function $\mathrm{ALLEQ}$ is in $L_+^{>1}$ since it is not correlated with a single bit.
To analyze
%***CHANGED "these" to
this
tribe-like construction we need the following.

\begin{proposition} \label{prop:general_tribes}
Let $g\colon\{-1,1\}^{r}\to\{-1,1\}$.  Consider a function $f\colon\{-1,1\}^{n}\to\{-1,1\}$ of the form
$$
f(x) =f_{b,r}(x)\colonequals g(x_1,\ldots,x_r) \vee g(x_{r+1},\ldots,x_{2r}) \vee \cdots \vee g(x_{(b-1)r+1},\ldots,x_{br}), \,\, br = n,
$$
and where $\P(g = 1) \leq 2^{-m}$ where $m \leq r$. Then
\begin{equation}\label{eight1}
\E f= 2(1 - (1-\P(g=1))^b)-1,
\end{equation}
\begin{equation}\label{eight2}
\max_{i=1,\ldots,n} I_i(f) \leq 2 \times 2^{-m}.
\end{equation}
One can choose $b$ so that
\begin{equation}\label{eight3}
|\E f | \leq 2^{-m+1}.
\end{equation}
\end{proposition}
\begin{proof}
Equation \eqref{eight1} is obvious, and \eqref{eight3} follows from the fact that
\[
0\leq\E[f_{b,r}-f_{b+1,r}] \leq2^{-m+1}.
\]
Equation \eqref{eight2} is also easy: for $x_i$ to be pivotal where $i\in\{dr+1,dr+2,\ldots(d+1)r\}$, we need that the $g$ value of the other $x_j$ in the block with $j\in\{dr+1,dr+2,\ldots,(d+1)r\}$, together with either $x_i = -1$ or $x_i = 1$ evaluate to $1$.
%To prove \eqref{eight4}, let $y_{1},\ldots,y_{b}\in\{-1,1\}$ and define $\mathrm{OR}(y_{1},\ldots,y_{b})\colonequals y_{1}\vee\cdots\vee y_{b}$.  Then write
%\begin{equation}\label{eight5}
%\mathrm{OR}(y)=\sum_{S\subset\{1,\ldots,b\}}\widehat{\mathrm{OR}}(S)\prod_{i\in S}y_{i}.
%\end{equation}
%Now, note that
%$f(x)=\mathrm{OR}(g(x_{1},\ldots,x_{r}),\ldots,g(x_{(b-1)r+1},\ldots,x_{br}))$, so write
%\begin{equation}\label{eight6}
%g(x_{1},\ldots,x_{r})=\sum_{S\subset\{1,\ldots,r\}}\widehat{g}(S)\prod_{i\in S}x_{i}.
%\end{equation}
% and substituting \eqref{eight6} into \eqref{eight5} proves \eqref{eight4}.
\end{proof}

We will also need the following fact
\begin{proposition} \label{prop:general_tribes2}
Consider a function of the form:
\[
f(x) = F(g_1(x_1,\ldots,x_{r}),g_2(x_{r+1},\ldots,x_{2r}) \ldots g_{b}(x_{(b-1)r+1},\ldots,x_{br})),\quad br=n,
\]
where  $\{g_j\}_{j=1}^{b}$ are Boolean functions all taking the values $\{0,1\}$ or all taking the values $\{-1,1\}$. Assume further that $g_j \in L_+^{>k}(\{-1,1\}^{r})$ for all $j=1,\ldots,b$. Then $f \in L_+^{>k}(\{-1,1\}^{n})$.
\end{proposition}

\begin{proof}
Since we can write $F$ as a multilinear polynomials of its binary inputs, it suffices to show that each product of a subset of the $g_i$ is in $L_+^{>k}$. By induction it suffices to show this for two functions which is immediate.
\end{proof}

We are particularly interested in the case where $g$ is an indicator of a linear code.
Recall that a {\em linear code} is a linear subspace of $\{0,1\}^n$, where we treat $\{0,1\}$ as the field of two elements.
The {\em minimal weight} $w(C)$ of a code $C$ is defined by
\[
w(C) = \min \{\vnorm{x}_{1} : 0 \neq x \in C\},
\]
where $\vnorm{(x_{1},\ldots,x_{n})}_{1}\colonequals\sum_{i=1}^{n}\abs{x_{i}}$ is the Hamming weight of $x$.
The {\em dual} code of $C$ denoted $C^\perp \subset\{0,1\}^n$ is given by
\[
C^{\perp} \colonequals \big\{ y \in \{0,1\}^n : \sum_{i=1}^{n} x_i y_i = 0 \mod 2, \; \forall\,x \in C \big\}.
\]
Given a code $C \subset \{0,1\}^n$, we will write $g_C\colon\{-1,1\}^{n}\to\{-1,1\}$ for the following Boolean function
% {-1,1}^{n}\to{0,1}^{n}, homomorphism
% (x_{1},\ldots,x_{n})\to ((1-x_{1})/2,...,(1-x_{n})/2)
% ab \to ((1-ab)/2)=?(1-a)/2 +(1-b)/2?  a=b=1, yes.  a=-1, b=-1, yes.  a=-1, b=1, yes
\[
g_C(x_1,\ldots,x_n)
\colonequals
\begin{cases}
1, & \mbox{ if } ((1-x_{1})/2,\ldots,(1-x_{n})/2) \in C \\
-1, & \mbox{ if } ((1-x_{1})/2,\ldots,(1-x_{n})/2) \notin C.
\end{cases}
\]
%\[
%g_C(x_1,\ldots,x_n)
%\colonequals
%\begin{cases}
%1, & \mbox{ if } ((-1)^{x_1},\ldots,(-1)^{x_n}) \in C \\
%-1, & \mbox{ if } ((-1)^{x_1},\ldots,(-1)^{x_n}) \notin C.
%\end{cases}
%\]
%https://www.cs.cmu.edu/~venkatg/teaching/codingtheory/notes/notes1.pdf
%https://www.cs.cmu.edu/~venkatg/teaching/codingtheory/notes/notes2.pdf
%
By the MacWilliams identities~\cite{MacWilliams:63}, see e.g.~\cite[Lemma 3.3]{khot06}  we have:
% MacWilliams, Jessie. "A Theorem on the Distribution of Weights in a Systematic Codes." Bell System Technical Journal 42.1 (1963): 79-94.
%{\bf Comment: Steve can't access the MacWilliams paper, unfortunately}
\begin{proposition} \label{prop:MacWilliams}
Let $C$ be a linear code. Then
$g_C \in L_+^{>k}$ if and only if  $w(C^{\perp}) > k$.
\end{proposition}
For example, for $C = \{(0,\ldots,0),(1 \ldots 1)\}$, we have
$g_C(x) = 1$ if and only if $x = \pm (1,\ldots,1)$, and the code $C^{\perp}$ consists of all codewords $x$ with $\vnorm{x}_{1}$ even, so $C^{\perp}$ has minimal weight $w(C^{\perp})=2$.

%*** Changed \gamma<1 to \gamma>1
%*** Changed 2^{-\gamma m} to 2^{-3m}
%*** Changed L_{+}^{m} to L_{+}^{>m}
\begin{proposition} \label{prop:tribek}
There exists a constant $\gamma > 1$ such that for every $m>0$, there
exists a function $g : \{-1,1\}^{\lceil \gamma m \rceil} \to \{-1,1\}$ with $g \in L_{+}^{>m}$
and $2^{-3 m} \leq \P[g = 1] \leq 2^{-m}$.
\end{proposition}
\begin{proof}
% H(\gamma)=-\gamma\log\gamma-(1-\gamma)\log(1-\gamma)
% for gamma small, H(\gamma)\approx\gamma
% known that 1-H(\gamma/2)\geq dim(C)\geq 1-H(\gamma)-o(1)
%
%*** CHANGED {-1,1}^{m'} to {0,1}^{m'}
The function $g$ will be constructed via the dual of a``good code.'' It is well known that good codes exist \cite{macwilliams77}.
Such (linear) codes $C \subset \{0,1\}^{m'}$ have the following properties (where $\delta$ is independent of $m'$).
\begin{itemize}
\item $(3/4)m'\geq\dim(C) \geq m'/4$,
\item $w(C) \geq \delta m'$, where $\delta > 0$.
\end{itemize}
We let $g \colonequals g_{C^{\perp}}$.  Then $\P[g = 1] = 2^{- \dim(C^{\perp})}$, so
\[
2^{-3m'/4}\leq\P[g = 1] \leq 2^{-m'/4},
\]
and by Proposition~\ref{prop:MacWilliams}, $g \in L_+^{> k}$ where
%*** Changed location of k from the left side to the right side of the following inequality
\[
w(C^{\perp \perp}) = w(C) \geq \delta m'=k.
\]
Setting $\gamma = \max(4,\delta^{-1})$, the proof follows.
%m'=4m=am
% bm=\delta m', b=\delta m'/m=\delta 4
%
% m=m'/4.
%
%alternately, choosing m=3m'/4
\end{proof}

%*** ADDED the following proof
Propositions~\ref{prop:general_tribes} and \ref{prop:tribek} are already enough to prove that Harper's inequality cannot be improved for tail spaces.
\begin{proof}[Proof of Theorem \ref{thm31}]
Let $b=1$ in Proposition~\ref{prop:general_tribes} and use $g$ from Proposition \ref{prop:tribek}.  Setting $n\colonequals\lceil\gamma m\rceil$ we get $g\colon\{-1,1\}^{n}\to\{-1,1\}$ with $g\in L_{+}^{>m}$, $\E g=2\P(g=1)-1$, $\max_{i=1,\ldots,n}I_{i}g\leq2\P(g=1)$.  Then, the function $f\colonequals (1+h)/2=1_{(h=1)}$ satisfies $f\colon\{-1,1\}^{n}\to\{0,1\}$, $\sum_{i=1}^{n}I_{i}f\leq 2n\P(g=1)\leq \gamma m\P(g=1)$, and $\E f=1/2+\E h/2=\P(g=1)$.  From Proposition \ref{prop:tribek}, $\P[g = 1] \leq 2^{-m}$.  That is,
$$\frac{\sum_{i=1}^{n}I_{i}f}{(\E f)\log(1/\E f)}\leq \frac{\gamma m}{\log(1/\E f)}\leq \gamma \frac{m}{m}=\gamma.$$
%so, the left side of Harper is Cm\P(g=1), and the right side of harper is P(g=1)log(1/P(g=1))
% 2^{-3m}\leq Ef\leq 2^{-m}, so 2^{m}\leq 1/Ef\leq 2^{3m}, so m\leq \log(1/E f)\leq 3m, so [3m]^{-1}\leq[\log(1/Ef)]^{-1}\leq m^{-1}
\end{proof}

%*** ADDED reference to prop:general tribes 2
Substituting $g$ from Proposition~\ref{prop:tribek} into Propositions~\ref{prop:general_tribes} and \ref{prop:general_tribes2}, and letting $n = m b$, where $b$ is chosen so that $\E[f]$ is as close to $0$ as possible
(so that $m = O(\log n)$), we obtain:
\begin{theorem}\label{thm10}
There exist a family of Boolean functions $f = f_n : \{-1,1\}^n \to \{-1,1\}$ such that
\begin{itemize}
\item
 $f \in L_+^{> \Omega(\log n)}(\{-1,1\}^{n})$.
\item
For all $i\in\{1,\ldots,n\}$, $I_i(f) \leq O((\log n) / n)$.
\item
$|\E f| \leq O((\log n) / n)$.
\end{itemize}
\end{theorem}
%note that g has 4m variables, this seems okay

We now wish to find similar functions that have zero mean.
\begin{corollary}
There exist a family of functions $g = g_n : \{-1,1\}^{2n} \to \{-1,0,1\}$ such that
\begin{itemize}
\item
 $g \in L^{> \Omega(\log n)}(\{-1,1\}^{2n})$.
\item
For all $i\in\{1,\ldots,n\}$, $I_i(g) \leq O(\log n / n)$.
\item
$\P[g =1] = 1/4 - O((\log n) /n)$, $\P[g =-1] = 1/4 - O((\log n) /n)$.
\end{itemize}
\end{corollary}

\begin{proof}
Let $f$ from Theorem \ref{thm10} and define
\[
g(x_1,\ldots,x_n,y_1,\ldots,y_n) \colonequals \frac{1}{2}(f(x_1,\ldots,x_n)-f(y_1,\ldots,y_n)).
\]
\end{proof}

With a little more work we can construct functions with the desired properties taking only values $0$ and $1$.  For this we note that Proposition~\ref{prop:tribek} implies the following:
\begin{corollary} \label{cor:tribek}
There exists a constant $\gamma > 1$ such that for every $n$, there
exists a function $g : \{-1,1\}^{\gamma n} \to \{0,1\}$ with $g \in L_{+}^{ > n}$
%*** CHANGED: positive to nonnegative
and $\P[g = 1] = 2^{-n-d}$ for some nonnegative integer $d$.
Moreover, $g$ has the following property:
For $y \in \{-1,1\}^{\gamma n}$, write $g_y(x) = g(y_1 x_1,\ldots, y_n x_n)$. Then for all
$y, y' \in \{-1,1\}^{ \gamma n}$ we either have $g_y = g_{y'}$ or the function $g_y g_{y'}$ is identically $0$.
\end{corollary}
\begin{proof}
%*** CHANGED Subscript notation from 1(h=1) to $1_{(h=1)}$
Let $h$ be the function from Proposition~\ref{prop:tribek} and let $g = 1_{(h=1)} = (h+1)/2$.
Then all the stated properties but the last one clearly hold if $\gamma$ is large enough. The last property follows from the fact that cosets of linear codes are either identical or disjoint.
\end{proof}

\begin{lemma}\label{lemma15}
The exists a constant $\gamma>1$, such that the following holds.
Let $0 \leq t < 2^{n}$, $t\in\Z$. Then there exists a function
$f : \{-1,1\}^{\gamma n} \to \{0,1\}$ such that $\E f = t/2^n$ and $f \in L_+^{> n}(\{-1,1\}^{\gamma n})$.
\end{lemma}

\begin{proof}
%*** ADDED: $d$ is a nonnegative integer
From Corollary~\ref{cor:tribek} in the case $t=1$ we can find a function in $L_+^{>n}$ and
$\E[f] = 2^{-n-d}$, where $d$ is a nonnegative integer. The general case follows by taking $h = \sum_i g_{y^i}$ where $y^i$ are chosen so that $g_{y^i} g_{y^j} = 0$ for $i \neq j$.
\end{proof}

\begin{theorem}
There exist a family of Boolean functions $G = G_n : \{-1,1\}^n \to \{-1,1\}$ such that
\begin{itemize}
\item
 $G \in L^{> \Omega(\log n)}$
\item
For all $i\in\{1,\ldots,n\}$, $I_i(G) \leq O((\log n) / n)$.
\end{itemize}
\end{theorem}

\begin{proof}
We revise the construction of Theorem \ref{thm10} as follows. Using Lemma \ref{lemma15}, choose $g_0,\ldots,g_b\colon\{-1,1\}^{\lceil \gamma m \rceil}\to\{-1,1\}$ all in $L_+^{>m}$.
Moreover, for $1 \leq i \leq b$, let $\P[g_i = 1] = 2^{-m}$ and for $i = 0$, let
$\P[g_0  = 1] = 4 \times 2^{-m}$.

We choose $b$ to be the largest integer so that
\[
\E f =(1 - (1-2^{-m})^b(1-2^{-m+2}))-1 > 0.
\]
and let $n = (b+1) \lceil \gamma m \rceil$.
Note that $m = O(\log n)$ and that
\[
0\leq\E f \leq 2^{-m}.
\]
%Also, there exist $d,r\in\N$ with $r\leq 4m$ such that
%\begin{equation}\label{seven1}
%\E f=d2^{-n},\qquad \E g_{0}=2^{-r}
%\end{equation}
%  Let $x_1,\ldots,x_n$ denote the coordinates of $f$.
% and let $g$ be a function of $x_1,\ldots,x_n,y_1,\ldots,y_{\gamma n}$.
%recall that m\leq r, rb=n
By Lemma \ref{lemma15}, let $h\colon\{-1,1\}^{\lceil \gamma n \rceil}\to\{0,1\}$ with
\begin{equation}\label{seven2}
2 \E h=\E f / \P[g_0 = 1]
\end{equation}
 and such that $h$ is in $L_{+}^{>n}(\{-1,1\}^{\gamma n})$.
 %Note that any $t\in\N$, $h$ exists such that $\E h=t2^{-n}$ by Lemma \ref{lemma15}.
%  So, from \eqref{seven1}, we have $(\E g_{0})^{-1}\E f=d2^{-n+r}$, so we can choose $t\colonequals d2^{r}$ so that $\E h=d2^{r-n}=(\E g_{0})^{-1}\E f$.  That is, \eqref{seven2} holds.

Let $G\colon\{-1,1\}^{\gamma n + n}\to\{-1,1\}$ be a function of the $x$ and $y$ given by:
%then h can be any multiple of 2^-n.  is this allowed?  what is m?
\[
G(x,y) \colonequals f(x) - 2\cdot g_{0}(x)\cdot h(y)
\]
Then clearly $G(x,y) \in L_{+}^{m}(\{-1,1\}^{\lceil \gamma n \rceil + n})$ and moreover $\E g = 0$ by \eqref{seven2}.  So we have $G\in L^{m}(\{-1,1\}^{\lceil \gamma n \rceil + n})$.    Finally, since $f(x)$ and $g_0(x)$ have all of their influences $O((\log n)/n)$ the same is true for all of
the $x$ variables in $g$. Moreover, a $y$ variables can be influential if and only if $g_0(x) = 1$. Therefore the influence of all of the $y$ variables is also $O((\log n) / n)$.
The proof follows.
\end{proof}

%\subsection{Infinite Spaces}\label{secinf}

%The proof of Theorem \ref{main} used the finiteness of the space $\Omega$ in a nonessential way.  Infinite spaces require a bit more care, so we have chosen the above presentation.  The only place where finiteness was used was in \eqref{ten10}. A suitable replacement for this inequality appears in \cite[Eq. (1.6)]{grigoryan12} or \cite[Eq. (2.7)]{ledoux99}, where it is shown that for all $t>0$, there exist functions $p_{t}\colon\Omega\times\Omega\to[0,\infty]$ such that, for any smooth functions $f,g\colon\Omega\to\R$, we have

%\begin{equation}\label{zero3}
%\E fLg=\lim_{t\to0^{+}}\frac{1}{2t}\int_{\Omega}\int_{\Omega}(f(x)-f(y))(g(x)-g(y))p_{t}(x,y)d\mu(x)d\mu(y).
%\end{equation}

%\medskip

\noindent{\textbf{Acknowledgement.}}
Thanks for Fedor Nazarov for sharing his complex interpolation arguments of the main theorem which strongly inspired the proof of Section \ref{secnaz}.  Thanks to Ryan O'Donnell for helpful discussions and references, particularly \cite{talagrand89} and \cite{janson97}.  Thanks also to Michel Ledoux, Camil Muscalu, Assaf Naor, and Bob Strichartz for helpful discussions.

\bibliographystyle{amsplain}
\bibliography{12162011}

\end{document}